\documentclass[a4paper,10pt]{amsart}
\usepackage{amsmath,amsthm,amssymb,amsfonts,enumerate,color,esint}
\usepackage[pdftex]{graphicx}

\newcommand{\set}[1]{\left\{#1\right\}}
\newcommand{\R}{\mathbb{R}}

\newcommand{\Dom}{\operatorname{Dom}}

\newcommand{\dive}{\operatorname{div}}
\newcommand{\osc}{\operatornamewithlimits{osc}}
\renewcommand{\H}{\mathcal{H}}
\newcommand{\dist}{\operatorname{dist}}

\newtheorem{thm}{Theorem}[section]
\newtheorem{prop}[thm]{Proposition}
\newtheorem{cor}[thm]{Corollary}
\newtheorem{lem}[thm]{Lemma}

\theoremstyle{definition}
\newtheorem{defn}[thm]{Definition}
\newtheorem{rem}[thm]{Remark}

\numberwithin{equation}{section}

\allowdisplaybreaks

\author[L. A. Caffarelli]{Luis A. Caffarelli}

\author[P. R. Stinga]{Pablo Ra\'ul Stinga}

\address{Department of Mathematics\\
The University of Texas at Austin\\
1 University Station, C1200, Austin\\
TX 78712-1202, United States of America}
\email{caffarel@math.utexas.edu, stinga@math.utexas.edu}

\thanks{The first author was supported by NSF Grant DMS-0654267. The second author was
supported by grant MTM2011-28149-C02-01 from Spanish Government}



\begin{document}

%%%%%%%%%%%%%%%%%%%%%%%%%%%%%%%%%%%%%%%%%%%%%%%%%%%%%%
\title[Fractional elliptic equations, Caccioppoli estimates and regularity]{Fractional elliptic equations, \\
 Caccioppoli estimates and regularity}
%%%%%%%%%%%%%%%%%%%%%%%%%%%%%%%%%%%%%%%%%%%%%%%%%%%%%%

%%%%%%%%%%%%%%%%%%%%%%%%%%%%%%%%%%%%%%%%%%%%%%%%%%%%%%
\begin{abstract}
Let $L=-\dive_x(A(x)\nabla_x)$ be a uniformly elliptic operator in divergence form in a bounded
domain $\Omega$.
We consider the fractional nonlocal equations
$$\begin{cases}
L^su=f,&\hbox{in}~\Omega,\\
u=0,&\hbox{on}~\partial\Omega,
\end{cases}\quad
\hbox{and}\quad
\begin{cases}
L^su=f,&\hbox{in}~\Omega,\\
\partial_Au=0,&\hbox{on}~\partial\Omega.
\end{cases}$$
Here $L^s$, $0<s<1$, is the fractional power of $L$ and $\partial_Au$ 
is the conormal derivative of $u$ with respect to the coefficients $A(x)$.
We reproduce Caccioppoli type estimates that allow us to develop the regularity theory.
Indeed, we prove interior and boundary Schauder regularity estimates
depending on the smoothness of the coefficients $A(x)$, the right hand side $f$ and the boundary of the domain.
Moreover, we establish estimates for fundamental solutions
in the spirit of the classical result by Littman--Stampacchia--Weinberger
and we obtain nonlocal integro-differential formulas for $L^su(x)$.
Essential tools in the analysis are the semigroup language approach and the extension problem.
\end{abstract}
%%%%%%%%%%%%%%%%%%%%%%%%%%%%%%%%%%%%%%%%%%%%%%%%%%%%%%

\maketitle

%\tableofcontents

%%%%%%%%%%%%%%%%%%%%%%%%%%%%%%%%%%%%%%%%%%%%%%%%%%%%%%
\section{Introduction}
%%%%%%%%%%%%%%%%%%%%%%%%%%%%%%%%%%%%%%%%%%%%%%%%%%%%%%

In a bounded Lipschitz domain $\Omega\subset\R^n$, $n\geq1$, we consider an elliptic operator
in divergence form
$$Lu=-\dive_x(A(x)\nabla_xu),$$
with boundary condition
$$\hbox{Dirichlet:}~u=0,\quad\hbox{or}\quad\hbox{Neumann:}~\partial_Au:=A(x)\nabla_xu\cdot\nu=0,
\quad\hbox{on}~\partial\Omega,$$
where $\nu$ is the exterior unit normal to $\partial\Omega$.
The coefficients are symmetric $A(x)=A^{ij}(x)=A^{ji}(x)$, $i,j=1,\ldots,n$,
 bounded and measurable in $\Omega$ and satisfy the uniform ellipticity condition
$\Lambda_1|\xi|^2\leq A(x)\xi\cdot\xi\leq\Lambda_2|\xi|^2$,
for all $\xi\in\R^n$ and almost every $x\in\Omega$,
for some ellipticity constants $0<\Lambda_1\leq\Lambda_2$.
In this paper we study interior and boundary regularity estimates for the fractional nonlocal problem
$L^su=f$ in $\Omega$, subject to the boundary conditions above,
in the cases when $f$ is H\"older continuous, see Theorems \ref{thm:interior Calpha}, 
\ref{thm:boundary Calpha} and \ref{thm:global Calpha Neumann}, and when $f$ is just $L^p$ integrable,
see Theorems \ref{thm:interior Lp} and \ref{thm:boundary Lp}. 
These estimates of course depend on the regularity of the coefficients $A$ and of the boundary of $\Omega$.
Our main tools are the semigroup language
approach as developed in \cite{Stinga} and the extension problem as introduced in \cite{Caffarelli-Silvestre}.
We obtain Caccioppoli type estimates that are combined with a compactness and approximation argument based on
the ideas of \cite{Caffarelli} to prove the regularity results.

Let us begin by considering the case of Dirichlet boundary condition.
By using the $L^2$-Dirichlet eigenvalues and eigenfunctions $(\lambda_k,\phi_k)_{k=0}^\infty$,
$\phi_k\in H^1_0(\Omega)$, of $L$ we can define the fractional powers
$L^su$, $0<s<1$, for $u$ in the domain $\Dom(L^s)\equiv\mathcal{H}^s$
(see Remark \ref{rem:H fancy en Hs}) in the natural way.
If $u(x)=\sum_{k=0}^\infty u_k\phi_k(x)$, $x\in\Omega$, then
$$L^su(x)=\sum_{k=0}^\infty\lambda_k^su_k\phi_k(x).$$
Observe that $u=0$ on $\partial\Omega$. Equivalently, we have the semigroup formula
\begin{equation}\label{abc}
L^su(x)=\frac{1}{\Gamma(-s)}\int_0^\infty\big(e^{-tL}u(x)-u(x)\big)\,\frac{dt}{t^{1+s}},
\end{equation}
where $\{e^{-tL}u\}_{t>0}$ is the heat diffusion semigroup generated by $L$
with Dirichlet boundary condition and $\Gamma$ is the Gamma function.
See Section \ref{section:2}. It is clear that
for $f$ in the dual space $\H^{-s}\equiv(\mathcal{H}^s)'$ there exists a unique solution $u\in\H^s$
to the fractional nonlocal equation
\begin{equation}\label{uma}
\begin{cases}
L^s u=f,& \hbox{in}~\Omega,\\
u=0, & \hbox{on}~\partial\Omega.
\end{cases}
\end{equation}
Starting from \eqref{abc} and by using the heat kernel for $e^{-tL}$ we are able to
obtain integro-differential formulas for $L^su(x)$ of the form
\begin{equation}\label{eq:formula puntual}
 \langle L^su,\psi\rangle=\int_{\Omega}\int_{\Omega}\big(u(x)-u(z)\big)\big(\psi(x)-\psi(z)\big)
 K_s(x,z)\,dx\,dz+\int_\Omega u(x)\psi(x)B_s(x)\,dx,
 \end{equation}
where $\psi\in\mathcal{H}^s$, see Theorem \ref{thm:puntual}.
Observe that this formula is parallel to the weak form interpretation of $Lu$ in $H^1_0(\Omega)$.
Estimates for the kernel $K_s(x,z)$ and the fundamental solution $G_s(x,z)$ of $L^s$
are contained in Theorems \ref{thm:estimaciones Ks},
\ref{thm:LSW} and \ref{thm:estimaciones Gs}. In particular, we show that the fundamental solution
satisfies the interior estimate
$$G_s(x,z)\sim\frac{1}{|x-z|^{n-2s}},\quad\hbox{for}~x,z\in\Omega,$$
when $n\neq 2s$, with a logarithmic estimate when $n=2s$.

Similarly, we can define the fractional powers of $L_N$, the operator $L$ subject
to Neumann boundary conditions.
In this case we use the Neumann eigenvalues and eigenfunctions
$(\mu_k,\varphi_k)_{k=0}^\infty$,  $\varphi_k\in H^1(\Omega)$, to define $L_N^su$ as
\begin{equation}\label{definition Neumann}
L_N^su(x)=\sum_{k=1}^\infty\mu_k^su_k\varphi_k(x).
\end{equation}
The formula in \eqref{abc} is
also valid for $L_N$ in place of $L$. Then we obtain the integro-differential formula
\begin{equation}\label{puntual Neumann}
 \langle L_N^su,\psi\rangle=\int_{\Omega}\int_{\Omega}\big(u(x)-u(z)\big)\big(\psi(x)-\psi(z)\big)
 K^N_s(x,z)\,dx\,dz,
 \end{equation}
 where the kernel $K^N_s(x,z)$ is given in terms of the heat kernel for $e^{-tL_N}$.
 Notice the difference between this formula and 
 the one in \eqref{eq:formula puntual} for the Dirichlet case. This is so because
 for Neumann boundary condition we have $e^{-tL_N}1\equiv1$, while 
 for Dirichlet condition $e^{-tL}1\neq1$.  Now if $\int_\Omega f\,dx=0$ then
  it follows that there exists a unique solution
 $u\in\Dom(L_N^s)=H^s(\Omega)$ to
\begin{equation}\label{uman}
 \begin{cases}
 L^su=f,&\hbox{in}~\Omega,\\
 \partial_Au=0,&\hbox{on}~\partial\Omega,
 \end{cases}
 \end{equation}
  with $\int_\Omega u\,dx=0$. For the details see Section \ref{section:7}.
 
 It is already known, see \cite{Stinga-Torrea},
 that the fractional operators \eqref{abc} can be described as Dirichlet--to--Neumann maps for an extension problem
 in the spirit of the extension problem for the fractional Laplacian on $\R^n$ of \cite{Caffarelli-Silvestre}.
In fact, let $U=U(x,y):\Omega\times(0,\infty)\to\R$ be the solution to
the degenerate elliptic equation with $A_2$ weight
 \begin{equation}\label{dos estrellitas}
 \begin{cases}
  \dive(y^aB(x)\nabla U)=0,& \hbox{in}~\Omega\times(0,\infty),\\
  U=0,& \hbox{on}~\partial\Omega\times[0,\infty),\\
  U(x,0)=u(x),& \hbox{on}~\Omega,
\end{cases}
\end{equation}
where
\begin{equation}\label{B}
B(x):=\Bigg[\begin{array}{cc}
A(x) & 0 \\
0 & 1
\end{array}\Bigg]\in\R^{n+1}\times\R^{n+1},\quad\hbox{and}\quad a:=1-2s\in(-1,1).
\end{equation}
Then, for $c_s=|\Gamma(-s)|/(4^s\Gamma(s))>0$,
$$-\frac{1}{2s}\lim_{y\to0^+}y^aU_y(x,y)=-\lim_{y\to0^+}
\frac{U(x,y)-U(x,0)}{y^{2s}}=c_sL^su(x),\quad x\in\Omega.$$
Moreover, there are explicit formulas for $U$ in terms of the semigroup $e^{-tL}$.
See Theorem \ref{thm:extension} and the comments before it.
When $A(x)=I$ and $\Omega=\R^n$ in \eqref{dos estrellitas}
we recover the extension problem for the fractional
Laplacian of \cite{Caffarelli-Silvestre}.
By replacing the second equation for $U$ in \eqref{dos estrellitas} by $\partial_AU(x,y)=0$
 for all $y\geq0$ we get the extension problem for the fractional operator $L^s_N$.
  
Fractional powers of elliptic operators as those above arise naturally in applications,
for instance, in nonlinear elasticity, probability and mathematical biology.
Consider for example the following thin obstacle problem for an elastic
membrane $U(x,y):\Omega\times(0,\infty)\to\R$ and an obstacle $\varphi:\Omega\to\R$
such that $\varphi\leq0$ on $\partial\Omega$:
$$\begin{cases}
U_{yy}-LU=0=\dive(B(x)\nabla U),&\hbox{in}~\Omega\times(0,\infty),\\
U(x,0)\geq\varphi(x),&\hbox{on}~\Omega,\\
U_y(x,0)\leq0,&\hbox{on}~\{U(x,0)=\varphi(x)\},\\
U_y(x,0)=0,&\hbox{on}~\{U(x,0)>\varphi(x)\}.
\end{cases}$$
We also require for the membrane $U$ to be at level zero (Dirichlet)
or to have zero flux (Neumann) on $\partial\Omega\times[0,\infty)$.
The classical case of the Signorini problem is when $A(x)=I$, so the membrane is a harmonic
function in $\Omega\times(0,\infty)$.
It is cleat that the solution $U(x,y)$ to the first equation above with boundary datum $u(x):=U(x,0)$
and Dirichlet (or Neumann) boundary condition on $\partial\Omega\times[0,\infty)$
is given by the Poisson semigroup generated by $L$ (or $L_N$):
$$U(x,y)=e^{-yL^{1/2}}u(x),\quad x\in\Omega,~y\geq0.$$
Observe that $U_y(x,y)=-L^{1/2}e^{-yL^{1/2}}u(x)$ (see \cite{Stinga}).
Therefore we readily see that the membrane solves the thin obstacle problem
if and only if its trace $u$ solves the fractional obstacle problem
$$\begin{cases}
u\geq\varphi,&\hbox{in}~\Omega,\\
L^{1/2}u\geq0,&\hbox{in}~\{u=\varphi\},\\
L^{1/2}u=0,&\hbox{in}~\{u>\varphi\},
\end{cases}$$
with $u=0$ (or $\partial_A u=0$) on $\partial\Omega$, see \cite{Stinga}. This obstacle problem for $L=-\Delta$ and $\Omega=\R^n$
was studied in \cite{Caffarelli-Salsa-Silvestre, Silvestre CPAM}.
Another application comes from the theory of stochastic processes.
It is known that there is a Markov process $Y_t$ having as generator
the fractional power $(-\Delta_D)^s$ of the Dirichlet Laplacian $-\Delta_D$ on $\Omega$.
Indeed,  we first kill the Wiener process $X_t$ at $\tau_\Omega$, 
the first exit time of $X_t$ from $\Omega$, and then we subordinate the killed
Wiener process with an $s$-stable subordinator $T_t$. Hence $Y_t=X_{T_t}$ is the desired process,
see for example \cite{Song-Vondracek} and references therein. For a semilinear problem
involving the fractional Dirichlet Laplacian see \cite{CDDS} and references therein.
By considering nonlocal chemical diffusion in the Keller--Segel model one is led to
a semilinear problem for the fractional Neumann Laplacian, see \cite{Stinga-Volzone}.
Finally, we mention that finite element approximations for the fractional problem \eqref{uma}
were studied in \cite{Nochetto-Otarola-Salgado} by using the extension problem.

We now present the interior regularity estimates.

\begin{thm}[Interior regularity for $f$ in $C^\alpha$]\label{thm:interior Calpha}
Assume that $\Omega$ is a bounded Lipschitz domain and that $f\in C^{0,\alpha}(\Omega)$,
for some $0<\alpha<1$. Let $u$ be a solution to \eqref{uma} or \eqref{uman}.
 \begin{enumerate}[$(1)$]
  \item Suppose that $0<\alpha+2s<1$ and that $A(x)$ is continuous in $\Omega$.
  Then $u\in C^{0,\alpha+2s}({\Omega})$ and
  $$[u]_{C^{0,\alpha+2s}({\Omega})}\leq C\big(\|u\|_{L^2(\Omega)}+[u]_{H^s(\Omega)}
  +\|f\|_{C^{0,\alpha}(\Omega)}\big).$$
   \item Suppose that $1<\alpha+2s<2$ and that $A(x)$ is in $C^{0,\alpha+2s-1}(\Omega)$.
   Then $u\in C^{1,\alpha+2s-1}({\Omega})$ and
  $$[u]_{C^{1,\alpha+2s-1}({\Omega})}\leq C\big(\|u\|_{L^2(\Omega)}+[u]_{H^s(\Omega)}
  +\|f\|_{C^{0,\alpha}(\Omega)}\big).$$
 \end{enumerate}
 The constants $C$ above depend only on ellipticity, $n$, $\Omega$, $\alpha$, $s$ and
 the modulus of continuity of $A(x)$.
\end{thm}

\begin{thm}[Interior regularity for $f$ in $L^p$]\label{thm:interior Lp}
Assume that $\Omega$ is a bounded Lipschitz domain and that $f\in L^p(\Omega)$,
for some $1<p<\infty$. Let $u$ be a solution to \eqref{uma} or \eqref{uman}.
 \begin{enumerate}[$(1)$]
 \item Suppose that $n/(2s)<p<n/(2s-1)^+$
 and that $A(x)$ is continuous in $\Omega$.  Then $u\in C^{0,\alpha}({\Omega})$,
 for $\alpha=2s-n/p\in(0,1)$, and
$$[u]_{C^{0,\alpha}(\Omega)}\leq C\big(\|u\|_{L^2(\Omega)}+[u]_{H^s(\Omega)}
+\|f\|_{L^p(\Omega)}\big).$$
\item Suppose that $s>1/2$, $p>n/(2s-1)$ and that $A(x)$ is in $C^{0,\alpha}(\Omega)$,
for $\alpha=2s-n/p-1\in(0,1)$. Then $u\in C^{1,\alpha}(\Omega)$ and
$$[u]_{C^{1,\alpha}(\Omega)}\leq C\big(\|u\|_{L^2(\Omega)}+[u]_{H^s(\Omega)}
+\|f\|_{L^p(\Omega)}\big).$$
 \end{enumerate}
The constants $C$ above depend only on ellipticity, $n$, $\Omega$, $\alpha$, $s$
and the modulus of continuity of $A(x)$.
\end{thm}

The results above should be compared with the classical regularity estimates for the fractional
Laplacian and with the classical Schauder estimates for divergence form elliptic operators.
If $(-\Delta)^su=f$ in $\R^n$ and $f\in C^\alpha$ then $u\in C^{\alpha+2s}$.
On the other hand, if $Lu=f$ and the coefficients $A(x)$ and the right hand side $f$ are in $C^\alpha$,
then $u\in C^{1,\alpha}$ in the interior. If the coefficients are just continuous and $f$ is in $L^p$,
for some $n/2<p<n$, then $u\in C^{2-n/p}$, while if $p>n$ and the coefficients are H\"older
continuous with exponent $\alpha=2-n/p-1$ then $u\in C^{1,\alpha}$ in the interior.
See Proposition \ref{Prop:Silvestre1} and \cite{Gilbarg-Trudinger, Han-Lin, Silvestre CPAM, Stein}.

Notice that in Theorems 
\ref{thm:interior Calpha} and \ref{thm:interior Lp} we require the coefficients to be continuous
in part (1) and H\"older continuous in part (2).
The idea behind these results
is to compare the solution $u$ with the solution of the equation with frozen coefficients.
In (1) we notice that $u-c$ is still a solution in the interior for any constant $c$,
so the regularity basically comes from the right hand side as in the case of
the fractional Laplacian. For part (2),
if $\ell$ is a linear function then $u-\ell$ is not a solution of the same equation. Then
H\"older regularity in the coefficients is needed in order to gain a decay in
the oscillation of the remainder error in the right hand side.

Next we establish the boundary regularity in the case of Dirichlet boundary condition.

\begin{thm}[Boundary regularity for $f$ in $C^\alpha$ -- Dirichlet]\label{thm:boundary Calpha}
Assume that $\Omega$ is a bounded domain and that $f\in C^{0,\alpha}(\overline{\Omega})$,
for some $0<\alpha<1$. Let $u$ be a solution to \eqref{uma}.
 \begin{enumerate}[$(1)$]
  \item Suppose that $0<\alpha+2s<1$, $\Omega$ is a $C^1$
  domain and that $A(x)$ is continuous in $\overline{\Omega}$. Then
  $$u(x)\sim \dist(x,\partial\Omega)^{2s}+v(x),\quad\hbox{for}~x~\hbox{close to}~\partial\Omega,$$
  where $v\in C^{0,\alpha+2s}(\overline{\Omega})$. Moreover, 
  $$[v]_{C^{0,\alpha+2s}(\overline{\Omega})}\leq C\big(1+\|u\|_{L^2(\Omega)}+[u]_{H^s(\Omega)}
  +\|f\|_{C^{0,\alpha}(\overline{\Omega})}\big).$$
   \item Suppose that $s\geq1/2$, $1<\alpha+2s<2$, $\Omega$ is a $C^{1,\alpha+2s-1}$ domain
   and that $A(x)$ is in $C^{0,\alpha+2s-1}(\overline{\Omega})$.
   If $s>1/2$ then
   $$u(x)\sim\dist(x,\partial\Omega)+v(x),\quad\hbox{for}~x~\hbox{close to}~\partial\Omega,$$
   where $v\in C^{1,\alpha+2s-1}(\overline{\Omega})$, with
   $$[v]_{C^{1,\alpha+2s-1}(\overline{\Omega})}\leq C\big(1+\|u\|_{L^2(\Omega)}+[u]_{H^s(\Omega)}
  +\|f\|_{C^{0,\alpha}(\overline{\Omega})}\big).$$
  If $s=1/2$ then
   $$u(x)\sim\dist(x,\partial\Omega)|\ln\dist(x,\partial\Omega)|+w(x),\quad\hbox{for}~x~\hbox{close to}~\partial\Omega,$$
   where $w\in C^{1,\alpha}(\overline{\Omega})$, with
   $$[w]_{C^{1,\alpha}(\overline{\Omega})}\leq C\big(1+\|u\|_{L^2(\Omega)}+[u]_{H^{1/2}(\Omega)}
  +\|f\|_{C^{0,\alpha}(\overline{\Omega})}\big).$$

 \end{enumerate}
 In both cases, if $f(x_0)=0$ for some $x_0\in\partial\Omega$, then $u(x_0)=v(x_0)$
 (resp. $u(x_0)=w(x_0)$)
 and $u$ has the same regularity as $v$ (resp. $w$) at $x_0\in\partial\Omega$.
 The constants $C$ above depend only on ellipticity, $n$, $\Omega$, $\alpha$, $s$ and
 the modulus of continuity of $A(x)$.
\end{thm}

The result above should be compared with the boundary regularity estimates for the fractional
Dirichlet Laplacian $(-\Delta_D^+)^s$ in the half space $\R^n_+$ contained in
Theorem \ref{Thm:Laplacian half space}. Observe that here an odd reflection can be performed
to compare with the global problem.

For the case of Neumann boundary condition the global regularity is 
the same as the interior regularity.

\begin{thm}[Global regularity for $f$ in $C^\alpha$ -- Neumann]\label{thm:global Calpha Neumann}
Assume that $\Omega$ is a bounded domain and that $f\in C^{0,\alpha}(\overline{\Omega})$,
for some $0<\alpha<1$. Let $u$ be a solution to \eqref{uman}.
 \begin{enumerate}[$(1)$]
  \item Suppose that $0<\alpha+2s<1$, $\Omega$ is a $C^1$
  domain and that $A(x)$ is continuous in $\overline{\Omega}$. Then
  $u\in C^{0,\alpha+2s}(\overline{\Omega})$ and 
  $$[u]_{C^{0,\alpha+2s}(\overline{\Omega})}\leq C\big(\|u\|_{L^2(\Omega)}+[u]_{H^s(\Omega)}
  +\|f\|_{C^{0,\alpha}(\overline{\Omega})}\big).$$
   \item Suppose that $s\geq1/2$, $1<\alpha+2s<2$, $\Omega$ is a $C^{1,\alpha+2s-1}$ domain
   and that $A(x)$ is in $C^{0,\alpha+2s-1}(\overline{\Omega})$.
   Then $u\in C^{1,\alpha+2s-1}(\overline{\Omega})$ and
   $$[u]_{C^{1,\alpha+2s-1}(\overline{\Omega})}\leq C\big(\|u\|_{L^2(\Omega)}+[u]_{H^s(\Omega)}
  +\|f\|_{C^{0,\alpha}(\overline{\Omega})}\big).$$
 \end{enumerate}
 The constants $C$ above depend only on ellipticity, $n$, $\Omega$, $\alpha$, $s$ and
 the modulus of continuity of $A(x)$.
\end{thm}

Again this Theorem should be compared with the boundary regularity estimates for the fractional
Neumann Laplacian $(-\Delta_N^+)^s$ in the half space $\R^n_+$, see 
Theorem \ref{Thm:Laplacian half space Neumann}. In this case even reflections can be used
to relate the problem with the global one.

Finally, for $L^p$ right hand side, both Dirichlet and Neumann cases have the same
regularity up to the boundary.

\begin{thm}[Boundary regularity for $f$ in $L^p$]\label{thm:boundary Lp}
Assume that $\Omega$ is a bounded domain and that $f\in L^p(\Omega)$,
for some $1<p<\infty$. Let $u$ be a solution to \eqref{uma} or \eqref{uman}.
 \begin{enumerate}[$(1)$]
 \item Suppose that $n/(2s)<p<n/(2s-1)^+$, $\Omega$ is a $C^1$ domain
 and that $A(x)$ is continuous in $\overline{\Omega}$.  Then $u\in C^{0,\alpha}(\overline{\Omega})$,
 for $\alpha=2s-n/p\in(0,1)$, and
$$[u]_{C^{0,\alpha}(\overline{\Omega})}\leq C\big(\|u\|_{L^2(\Omega)}+[u]_{H^s(\Omega)}
+\|f\|_{L^p(\Omega)}\big).$$
\item Suppose that $s>1/2$, $p>n/(2s-1)$, $\Omega$ is a $C^{1,\alpha}$ domain
 and that $A(x)$ is in $C^{0,\alpha}(\Omega)$,
for $\alpha=2s-n/p-1\in(0,1)$. Then $u\in C^{1,\alpha}(\overline{\Omega})$ and
$$[u]_{C^{1,\alpha}(\overline{\Omega})}\leq C\big(\|u\|_{L^2(\Omega)}+[u]_{H^s(\Omega)}
+\|f\|_{L^p(\Omega)}\big).$$
 \end{enumerate}
The constants $C$ above depend only on ellipticity, $n$, $\Omega$, $\alpha$, $s$
and the modulus of continuity of $A(x)$.
\end{thm}

We recall that if $Lu=f$ and the coefficients
 $A(x)$ and the right hand side $f$ are in $C^\alpha$ up to the boundary of $\Omega$
then $u\in C^{1,\alpha}$ up to the boundary.
If the coefficients are just continuous up to the boundary and $f$ is in $L^p$
for some $n/2<p<n$ then $u$ is globally in $C^{2-n/p}$, while if $p>n$ and the coefficients are H\"older
continuous up to the boundary with exponent $\alpha=2-n/p-1$ then $u\in C^{1,\alpha}$ up to the boundary.
See \cite{Gilbarg-Trudinger, Han-Lin}.

Some bounds at the boundary for the fractional Dirichlet Laplacian
when $f$ is just bounded and $\Omega$ is smooth where obtained in \cite{CDDS}.
For the particular case $s=1/2$ in a smooth domain with a right hand side vanishing 
at the boundary see also \cite{Cabre-Tan}.
The regularity estimates for the Neumann case generalize the results for the
fractional Neumann Laplacian $(-\Delta_N)^{1/2}$ obtained in \cite[Theorem~3.5]{Stinga-Volzone}.
When the coefficients $A(x)$ and the domain $\Omega$ are smooth, the $L^p$--domain and regularity
of fractional powers of strongly elliptic operators was considered in \cite{Seeley1,Seeley2},
see also the recent preprint \cite{Grubb pre}\footnote{We are grateful to Gerd Grubb for several interesting comments about the smooth case.}. 
For the fractional Laplacian on $\R^n$ we know that
 the unique solution $u\in H^s(\R^n)$
of the Dirichlet problem
$$\begin{cases}
(-\Delta)^su=1,&\hbox{in}~B_1,\\
u=0,&\hbox{in}~\R^n\setminus B_1,
\end{cases}$$
is given by $u(x)=c_{n,s}(1-|x|^2)_+^s$, see \cite{Getoor}. Thus in general $u$ is globally
in $C^s$ but not in any $C^\alpha$ for $\alpha>s$, see also \cite{Ros-Oton-Serra}.
For this case the boundary regularity in fractional Sobolev spaces and
in H\"older spaces on smooth domains was studied in \cite{Grubb}.

Throughout the paper we will mainly focus on the case of the Dirichlet boundary condition.
We explain only in Section \ref{section:7} how the case of the Neumann condition works by
pointing out the main differences with the Dirichlet case.
In Section \ref{section:2} we define in a precise way the
fractional operator $L^s$. By using the heat semigroup $e^{-tL}$ and \eqref{abc}
we obtain the integro-differential
formula \eqref{eq:formula puntual}, with estimates on the kernel.
The extension problem is explained. We also include in this section
the estimates for the fundamental solutions and comment about the
Harnack inequality of \cite{Stinga-Zhang} and the De Giorgi--Nash--Moser theory for
the case of bounded measurable coefficients. Section \ref{section:3} contains a Caccioppoli inequality 
that we use to prove an approximation lemma via a compactness argument.
Here we also prove a trace inequality on balls with explicit dependence on
the radius that will be useful to prove regularity. Then Section \ref{section:4}
is devoted to the proof of the interior regularity results (Dirichlet case).
The case of the fractional Dirichlet Laplacian in a half space
is studied in detail in Section \ref{section:5}. We collect in Section \ref{section:6} the proof of the
boundary estimates for the Dirichlet case.
 
\smallskip 
 
\noindent\textbf{Notation.} The notation we will use in this paper is the following.
The upper half space is given by
$\R^n_+:=\{(x',x_n)\in\R^n:x'\in\R^{n-1},x_n>0\}$. For the extension problem
we use the notation $\R^{n+1}_+=\{(x,y)\in\R^{n+1}:x\in\R^{n},y>0\}$. We usually write $X=(x,y)\in\R^{n+1}_+$.
 For $x_0\in\R^n$ and $r>0$ we denote
\begin{align*}
	B_r(x_0) &= \{x\in\R^n:|x-x_0|<r\},\\
	B_r^+(x_0) &= B_r(x_0)\cap\R^n_+,\\
	B_r(x_0)^* &= B_r(x_0)\times(0,r)\subset\R^{n+1}_+,\\
	B_r^+(x_0)^*&=B_r^+(x_0)\times(0,r)\subset\R^{n}_+\times(0,\infty).
\end{align*}
We will just put $B_r$, $B_r^+$, etc. whenever $x_0=0$. The letters $C$, $c$ and $d$
will denote positive constants that may change at each occurrence. We will add subscripts to them
whenever we want to stress their dependence on other constants, domains, etc.
The matrix $B(x)$ and the parameter $a$ are given by \eqref{B}. 
The notation $\dive$ and $\nabla$
stand for the divergence and the gradient with respect to the variable $X=(x,y)\in\Omega\times(0,\infty)$.

%%%%%%%%%%%%%%%%%%%%%%%%%%%%%%%%%%%%%%%%%%%%%%%%%%%%%%
\section{Fractional divergence form elliptic operators}\label{section:2}
%%%%%%%%%%%%%%%%%%%%%%%%%%%%%%%%%%%%%%%%%%%%%%%%%%%%%%

Throughout this section, unless explicitly stated, $\Omega$ will be a bounded Lipschitz domain of
$\R^n$ and the matrix of coefficients $A(x)$ will be uniformly elliptic, bounded and measurable.

%%%%%%%%%%%%%%%%%%%%%%%%%%%%%%%%%%%%%%%%%%%%%%%%%%%%%%
\subsection{Definition of $L^s$}
%%%%%%%%%%%%%%%%%%%%%%%%%%%%%%%%%%%%%%%%%%%%%%%%%%%%%%

The operator $L$ is nonnegative and selfadjoint in
the Sobolev space $H^1_0(\Omega)$. Therefore there exists an orthonormal basis of $L^2(\Omega)$
consisting of eigenfunctions $\phi_k\in H^1_0(\Omega)$, $k=0,1,2,\dots$,
that correspond to eigenvalues $0<\lambda_0<\lambda_1\leq\lambda_2\leq\cdots\nearrow\infty$.
Let us define the domain $\mathcal{H}^s\equiv\Dom(L^s)$ of the fractional operator $L^s$, $0<s<1$, as the
Hilbert space of functions
$$u=\sum_{k=0}^\infty u_k\phi_k=\sum_{k=0}^\infty \langle u,\phi_k \rangle_{L^2(\Omega)}\phi_k\in L^2(\Omega),$$
with inner product
$$\langle u,\psi\rangle_{\mathcal{H}^s}^2:=\sum_{k=0}^\infty \lambda_k^su_kd_k,$$
where $\psi=\sum_{k=0}^\infty d_k\phi_k\in\mathcal{H}^s$.
Observe that for some positive constant $C$ we have
$\|u\|_{L^2(\Omega)}\leq C\langle u,u\rangle_{\mathcal{H}^s}
=C\|u\|_{\mathcal{H}^s}$, for $u\in\mathcal{H}^s$,
so that $\langle\cdot,\cdot\rangle_{\mathcal{H}^s}$ defines indeed an inner product in $\mathcal{H}^s$.
For $u\in\mathcal{H}^s$, let $L^su$ be the element in the dual space $\mathcal{H}^{-s}:=(\mathcal{H}^s)'$ given by
$$L^su=\sum_{k=0}^\infty \lambda_k^su_k\phi_k,\quad\hbox{in}~\mathcal{H}^{-s}.$$
Namely,
$$\langle L^su,\psi\rangle=\sum_{k=0}^\infty \lambda_k^su_kd_k=\langle u,\psi\rangle_{\mathcal{H}^s},$$
where $\langle\cdot,\cdot\rangle$ denotes the paring between $\mathcal{H}^s$ and $\mathcal{H}^{-s}$.
By the Riesz representation theorem any functional $f\in\H^{-s}$ can be
written as $f=\sum_{k=0}^\infty f_k\varphi_k$  in $\H^{-s}$, where the coefficients $f_k$ 
satisfy $\sum_{k=0}^\infty \lambda_k^{-s}f_k^2<\infty$.
With these definitions and observations, if $f=\sum_{k=0}^\infty f_k\phi_k\in\H^{-s}$ then the unique solution
$u\in\mathcal{H}^s$ to the Dirichlet problem \eqref{uma}
is given by
$u=\sum_{k=0}^\infty\lambda_k^{-s}f_k\phi_k\in\mathcal{H}^s$.
More generally, if $f\in\H^r$ for $r\geq0$ (here $\H^0:=L^2(\Omega)$), then 
there exists a unique solution $u\in\H^{r+2s}$.

\begin{rem}\label{rem:H fancy en Hs}
We use the following notation:
 \begin{equation}\label{fractional Sobolev spaces}
 H^s:=\begin{cases}
 H^s(\Omega),&\hbox{when}~0<s<1/2,\\
 H^{1/2}_{00}(\Omega),&\hbox{when}~s=1/2,\\
 H^s_0(\Omega),&\hbox{when}~1/2<s<1.
 \end{cases}
 \end{equation}
 The spaces $H^s(\Omega)$ and $H^s_0(\Omega)$, $s\neq1/2$, are 
 the classical fractional Sobolev spaces given by
 the completion of $C^\infty_c(\Omega)$ under the norm
 $$\|u\|_{H^s(\Omega)}^2=\|u\|_{L^2(\Omega)}^2+[u]_{H^s(\Omega)}^2,$$
 where
$$[u]_{H^s(\Omega)}^2= \int_\Omega\int_\Omega
 \frac{(u(x)-u(z))^2}{|x-z|^{n+2s}}\,dx\,dz,\quad0<s<1.$$
 The space $H^{1/2}_{00}(\Omega)$ is the Lions--Magenes space which consists of
 functions $u$ in $L^2(\Omega)$ such that $[u]_{H^{1/2}(\Omega)}<\infty$ and
 $$\int_{\Omega}\frac{u(x)^2}{\operatorname{dist}(x,\partial\Omega)}\,dx<\infty.$$
 See \cite[Chapter~1]{Lions-Magenes}, also \cite[Section~2]{Nochetto-Otarola-Salgado}
 for a discussion. The norm in any of these spaces is denoted by $\|\cdot\|_{H^s}$.
 We will later see, by using the extension problem, that in fact we have
 $\mathcal{H}^s=H^s$  as Hilbert spaces.
\end{rem}

%%%%%%%%%%%%%%%%%%%%%%%%%%%%%%%%%%%%%%%%%%%%%%%%%%%%%%
\subsection{Heat semigroup and pointwise formula}
%%%%%%%%%%%%%%%%%%%%%%%%%%%%%%%%%%%%%%%%%%%%%%%%%%%%%%

Given a function $u=\sum_{k=0}^\infty u_k\phi_k$ in $L^2(\Omega)$, the weak
 solution $v(x,t)$ to the heat equation
for $L$ with Dirichlet boundary condition
$$\begin{cases}
 v_t=-Lv,&\hbox{in}~\Omega\times(0,\infty),\\
 v(x,t)=0,&\hbox{on}~\partial\Omega\times[0,\infty),\\
 v(x,0)=u(x),&\hbox{on}~\Omega,
\end{cases}$$
is given by
$$v(x,t)\equiv e^{-tL}u(x)=\sum_{k=0}^\infty e^{-t\lambda_k}u_k\phi_k(x),$$
in the sense that, for every test function $\psi=\sum_{k=0}^\infty d_k\phi_k\in H^1_0(\Omega)$,
$$\langle e^{-tL}u,\psi\rangle_{L^2(\Omega)}=
\sum_{k=0}^\infty e^{-t\lambda_k}u_kd_k.$$
In particular, $e^{-tL}u\in L^2((0,\infty);H^1_0(\Omega))\cap C([0,\infty);L^2(\Omega))$,
and $\partial_te^{-tL}u\in L^2((0,\infty);H^{-1}(\Omega))$.

\begin{lem}\label{lem:semi}
 Let $u\in\H^s$. Then
 $$L^su=\frac{1}{\Gamma(-s)}\int_0^\infty\big(e^{-tL}u-u\big)\,\frac{dt}{t^{1+s}},
 \quad\hbox{in}~\H^{-s}.$$
 More precisely, if $\psi\in\H^s$ then
 \begin{equation}\label{Ls with semigroup}
 \langle L^su,\psi\rangle=\frac{1}{\Gamma(-s)}\int_0^\infty\big(\langle 
 e^{-tL}u,\psi\rangle_{L^2(\Omega)}-\langle u,\psi\rangle_{L^2(\Omega)}\big)\,\frac{dt}{t^{1+s}}.
 \end{equation}
\end{lem}

\begin{proof}
 We have the following numerical formula with the Gamma function:
 $$\lambda^s=\frac{1}{\Gamma(-s)}\int_0^\infty\big(e^{-t\lambda}-1\big)\,\frac{dt}{t^{1+s}},
 \quad\hbox{for}~\lambda>0,~0<s<1.$$
 Then, if $\psi=\sum_{k=0}^\infty d_k\phi_k$,
 \begin{align*}
  \langle L^su,\psi\rangle &= \sum_{k=0}^\infty \lambda_k^su_kd_k
  = \frac{1}{\Gamma(-s)}\sum_{k=0}^\infty \int_0^\infty\big(e^{-t\lambda_k}u_kd_k-u_kd_k\big)\,\frac{dt}{t^{1+s}}\\
  &= \frac{1}{\Gamma(-s)}\int_0^\infty\left(\sum_{k=0}^\infty e^{-t\lambda_k}u_kd_k-\sum_{k=0}^\infty 
  u_kd_k\right)\frac{dt}{t^{1+s}},
 \end{align*}
 which is the desired formula. The last identity follows from Fubini's theorem, since $u,\psi\in\H^s$.
\end{proof}

Let $W_t(x,z)$ be the heat kernel for $L$ with Dirichlet boundary condition, that is,
\begin{equation}\label{eq:heat kernel}
W_t(x,z)=\sum_{k=0}^\infty e^{-t\lambda_k}\phi_k(x)\phi_k(z)=W_t(z,x),\quad t>0,~x,z\in\Omega.
\end{equation}
We have that $W_t(x,z)\geq0$ (see \cite{Davies}) and that if $u,\psi\in L^2(\Omega)$ then
$$\langle e^{-tL}u,\psi\rangle_{L^2(\Omega)}=\int_\Omega\int_\Omega W_t(x,z)
u(z)\psi(x)\,dz\,dx=\langle u,e^{-tL}\psi\rangle_{L^2(\Omega)},\quad t\geq0.$$

\begin{thm}[Pointwise/energy formula]\label{thm:puntual}
 Let $u,\psi\in\H^s$. Then \eqref{eq:formula puntual} holds, where
 \begin{equation}\label{eq:kernel Ks}
 0\leq K_s(x,z):=\frac{1}{2|\Gamma(-s)|}\int_0^\infty W_t(x,z)\,\frac{dt}{t^{1+s}}
 \leq\frac{c_{n,s}}{|x-z|^{n+2s}},\quad x\neq z,
 \end{equation}
 and
 \begin{equation}\label{eq:funcion Bs}
 B_s(x):=\frac{1}{2|\Gamma(-s)|}\int_0^\infty\big(1-e^{-tL}1(x)\big)\,\frac{dt}{t^{1+s}}\geq0.
 \end{equation}
\end{thm}

\begin{proof}
 By plugging the heat kernel into \eqref{Ls with semigroup},
 \begin{align*}
  \Gamma(-s)&\langle L^su,\psi\rangle = \int_0^\infty\int_\Omega\left[\int_\Omega
  W_t(x,z)u(x)\psi(z)\,dx-u(z)\psi(z)\right]dz\,\frac{dt}{t^{1+s}} \\
  &= \int_0^\infty\int_\Omega\left[\int_\Omega
  W_t(x,z)\big(u(x)-u(z)\big)\psi(z)\,dx+u(z)\psi(z)\left(\int_\Omega
  W_t(x,z)\,dx-1\right)\right]dz\,\frac{dt}{t^{1+s}} \\
  &= \int_0^\infty\int_\Omega\int_\Omega
  W_t(x,z)\big(u(x)-u(z)\big)\psi(z)\,dx\,dz\,\frac{dt}{t^{1+s}} \\
  &\quad +\int_0^\infty\int_\Omega u(z)\psi(z)\left(e^{-tL}1(z)-1\right)\,dz\,\frac{dt}{t^{1+s}}=:I.
 \end{align*}
 By exchanging the roles of $x$ and $z$ and using the symmetry of the heat kernel, we also have that
 $$I=-\int_0^\infty\int_\Omega\int_\Omega
  W_t(x,z)\big(u(x)-u(z)\big)\psi(x)\,dx\,dz\,\frac{dt}{t^{1+s}} 
+\int_0^\infty\int_\Omega u(z)\psi(z)\left(e^{-tL}1(z)-1\right)\,dz\,\frac{dt}{t^{1+s}}.$$
Therefore, by adding both identities for $I$,
\begin{equation}\label{casi}
\begin{aligned}
 2|\Gamma(-s)|\langle L^su,\psi\rangle &= \int_0^\infty\int_\Omega\int_\Omega
  W_t(x,z)\big(u(x)-u(z)\big)\big(\psi(x)-\psi(z))\,dx\,dz\,\frac{dt}{t^{1+s}} \\
  &\qquad +\int_0^\infty\int_\Omega u(z)\psi(z)\left(1-e^{-tL}1(z)\right)\,dz\,\frac{dt}{t^{1+s}}.
\end{aligned}
\end{equation}
To reach the final expression with the kernel $K_s(x,z)$ and the function $B_s(x)$
we need to interchange the order of integration in \eqref{casi}. 
The estimate for the kernel $K_s(x,z)$ is contained in Theorem \ref{thm:estimaciones Ks}.
Since $u,\psi\in H^s$ (see Remark \ref{rem:H fancy en Hs}) it follows that
Fubini's theorem can be applied to the first term in the right hand side of \eqref{casi}.
For the second term in \eqref{casi}, take $\psi=u$. Observe that 
$0\leq e^{-tL}1(z)\leq 1$, which
follows from the maximum principle. This and the fact that 
$K_s(x,z)\geq0$ imply in \eqref{casi} that 
\begin{align*}
	0 &\leq \int_0^\infty\int_\Omega|u(z)|^2\big(1-e^{-tL}1(z)\big)\,dz\,\frac{dt}{t^{1+s}} \\
	&=2|\Gamma(-s)|\langle L^su,u\rangle-\int_\Omega\int_\Omega\big(u(x)-u(z)\big)^2K_s(x,z)\,dx\,dz
	\leq 2|\Gamma(-s)|\|u\|_{\mathcal{H}^s}<\infty.
\end{align*}
Then, by Fubini's theorem,
$$0\leq\int_0^\infty\int_\Omega|u(z)|^2\big(1-e^{-tL}1(z)\big)\,dz\,\frac{dt}{t^{1+s}}=
\int_\Omega|u(z)|^2B_s(z)\,dz<\infty,$$
with $B_s(z)$ as in the statement. The same is true for when we replace $u$ by $\psi$ and by
$u-\psi$. Thus, by writing $u\psi=\frac{1}{2}(u^2+\psi^2-(u-\psi)^2)$, it follows
that we can apply Fubini's theorem to the second term of \eqref{casi}.
\end{proof}

\begin{thm}[Estimates for $K_s(x,z)$]\label{thm:estimaciones Ks}
Let $K_s(x,z)\geq0$ be the kernel in \eqref{eq:kernel Ks}.
\begin{enumerate}[$(1)$]
\item If the coefficients $A(x)$ are bounded and measurable then
$$K_s(x,z)\leq\frac{c_{n,s}}{|x-z|^{n+2s}},\quad x,z\in\Omega,~x\neq z.$$
\item If the coefficients $A(x)$ are bounded and measurable in $\Omega=\R^n$ then
$$K_s(x,z)\sim\frac{c_{n,s}}{|x-z|^{n+2s}},\quad x,z\in\R^n,~x\neq z.$$
In this case the function $B_s(x)$ of \eqref{eq:funcion Bs} is identically zero.
\item If the coefficients $A(x)$ are H\"older continuous in $\Omega$ with exponent $\alpha\in(0,1)$
 then there exist positive constants $c$ and $\eta\leq 1\leq\rho$
depending only on $n$, $\alpha$, $\Omega$ and ellipticity, with $c$ depending also on $s$, such that
$$c^{-1}\min\bigg(1,\frac{\phi_0(x)\phi_0(z)}{|x-z|^{2\eta}}\bigg)\frac{1}{|x-z|^{n+2s}}
\leq K_s(x,z)\leq c\min\bigg(1,\frac{\phi_0(x)\phi_0(z)}{|x-z|^{2\rho}}\bigg)\frac{1}{|x-z|^{n+2s}},$$
where $\lambda_0$ and $\phi_0$ are the first eigenvalue and the first eigenfunction of $L$. Here, for some constant 
$C>0$ depending on $\alpha$, $n$, $\Omega$ and ellipticity,
$$C^{-1}\dist(x,\partial\Omega)^\rho\leq\phi_0(x)\leq C\dist(x,\partial\Omega)^\eta.$$
\item Under the hypothesis of $(3)$, if in addition
$\Omega$ is a $C^{1,\gamma}$ domain for some $0<\gamma<1$, then
the estimate in $(3)$ is true for $\eta=\rho=1$ and $c$ depending also on $\gamma$.
In particular, the estimate holds when $L^s=(-\Delta_D)^s$,
the fractional Dirichlet Laplacian in a $C^{1,\gamma}$ domain.
\end{enumerate}
\end{thm}

\begin{proof}
We use the following known estimates for the heat kernel \eqref{eq:heat kernel} and
then integrate in $t$ in \eqref{eq:kernel Ks} via the change of variables $r=|x-z|^2/t\in(0,\infty)$.

\noindent$(1)$ In this case there exist constants $C,c>0$ depending on ellipticity, $n$ and $\Omega$ such that
$$W_t(x,z)\leq C\frac{e^{-|x-z|^2/(ct)}}{t^{n/2}},$$
for all $x,z\in\Omega$, $t>0$, see \cite[p.~89]{Davies}, also \cite{Auscher}.

\noindent$(2)$ For the case of bounded measurable coefficients in the whole space, the result of
Aronson \cite{Aronson} establishes that for some positive constants $c_1,\ldots, c_4$ depending
on ellipticity and $n$,
$$c_1\frac{e^{-|x-z|^2/(c_2t)}}{t^{n/2}}\leq W_t(x,z)\leq c_3\frac{e^{-|x-z|^2/(c_4t)}}{t^{n/2}},$$
for all $x,z\in\R^n$, $t>0$. See also \cite[p.~97]{Davies}. Moreover, in this case we have $e^{-tL}1(x)
\equiv 1$, so $B_s(x)\equiv0$.

\noindent$(3)$ Under these hypotheses it is proved in \cite[Theorem~2.2]{Riahi} that 
there exists positive constants $\eta\leq1\leq\rho$ and $c,c_1,c_2$ depending
only on $n,\alpha,\Omega$ and ellipticity such that 
\begin{align*}
c^{-1}\min\bigg(1,\frac{\phi_0(x)\phi_0(z)}{1\wedge t^\eta}\bigg)e^{-\lambda_0t}
\frac{e^{-c_1|x-z|^2/t}}{1\wedge t^{n/2}}& \leq W_t(x,z) \\
&\leq c^{-1}\min\bigg(1,\frac{\phi_0(x)\phi_0(z)}{1\wedge t^\rho}\bigg)e^{-\lambda_0t}
\frac{e^{-c_2|x-z|^2/t}}{1\wedge t^{n/2}},
\end{align*}
for all $x,z\in\Omega$, $t>0$. The behavior of $\phi_0$ is also known, see \cite[(1.2)]{Riahi}.

\noindent$(4)$ This follows from the fact that in the heat kernel estimate written in $(3)$ 
above we can take $\eta=\rho=1$, see \cite[Remark~1,~p.~123]{Riahi}.
\end{proof}

%%%%%%%%%%%%%%%%%%%%%%%%%%%%%%%%%%%%%%%%%%%%%%%%%%%%%%
\subsection{Extension problem}
%%%%%%%%%%%%%%%%%%%%%%%%%%%%%%%%%%%%%%%%%%%%%%%%%%%%%%

We particularize to our situation the extension problem of Stinga--Torrea \cite{Stinga-Torrea}, which is
in turn a generalization
of the Caffarelli--Silvestre extension problem of \cite{Caffarelli-Silvestre}.
Let us explain the details, which can be found in \cite{Stinga, Stinga-Torrea}.

Let $u\in\H^s$. Consider the solution $U=U(x,y):\Omega\times[0,\infty)\to\R$ to the extension problem
\begin{equation}\label{una estrellita}
 \begin{cases}
  -LU+\frac{a}{y}U_y+U_{yy}=0,& \hbox{in}~\Omega\times(0,\infty),\\
  U(x,y)=0,& \hbox{on}~\partial\Omega\times[0,\infty),\\
  U(x,0)=u(x),& \hbox{on}~\Omega.
\end{cases}
\end{equation}
The equation above is in principle understood in the sense that
$U$ belongs to $C^\infty((0,\infty);H^1_0(\Omega))\cap C([0,\infty);L^2(\Omega))$ and
$$\int_{\Omega}A(x)\nabla_xU(x,y)\nabla_x\eta(x)\,dx=\int_\Omega\big(\tfrac{a}{y}U_y+U_{yy}\big)\eta(x)\,dx,
\quad\hbox{for each}~y>0,$$
for any test function $\eta\in H^1_0(\Omega)$. The boundary conditions in $y$ read
$$\lim_{y\to0^+}U(x,y)=u(x)~\hbox{in}~L^2(\Omega),\quad
\hbox{and}\quad\lim_{y\to\infty}U(x,y)=0~\hbox{weakly in}~L^2(\Omega).$$
Notice that problem \eqref{una estrellita} can also be written in divergence form as \eqref{dos estrellitas}.
It was shown in \cite{Stinga, Stinga-Torrea} that if $u=\sum_{k=0}^\infty u_k\phi_k$ then the solution to this problem
is given by
\begin{equation}\label{discovery channel}
U(x,y)=y^s\frac{2^{1-s}}{\Gamma(s)}\sum_{k=0}^\infty \lambda_k^{s/2}u_k
\mathcal{K}_s(\lambda_k^{1/2}y)\phi_k(x),
\end{equation}
where $\mathcal{K}_s$ is the modified Bessel function of the second kind
and parameter $s$. Equivalently,
\begin{equation}\label{formula con el semigrupo}
\begin{aligned}
	U(x,y) &= \frac{y^{2s}}{4^s\Gamma(s)}\int_0^\infty e^{-y^2/(4t)}e^{-tL}u(x)\,\frac{dt}{t^{1+s}} \\
	 &= \frac{1}{\Gamma(s)}\int_0^\infty e^{-y^2/(4t)}e^{-tL}(L^su)(x)\,\frac{dt}{t^{1-s}}.
\end{aligned}
\end{equation}
By using the heat kernel one can show that
\begin{equation}\label{Poisson formula}
U(x,y)=\int_\Omega P_y^s(x,z)u(z)\,dz,
\end{equation}
where the Poisson kernel $P_y^s(x,z)$ is given by
\begin{equation}\label{Poisson kernel}
P_y^s(x,z)=\frac{y^{2s}}{4^s\Gamma(s)}\int_0^\infty e^{-y^2/(4t)}W_t(x,z)\,\frac{dt}{t^{1+s}}.
\end{equation}
In addition, by letting $c_s=\frac{\Gamma(1-s)}{4^{s-1/2}\Gamma(s)}>0$, we have
\begin{equation}\label{Neumann condition}
-\lim_{y\to0^+}y^aU_y(x,y)=c_sL^su,\quad\hbox{in}~\mathcal{H}^{-s}.
\end{equation}

It is easy to show, by using the representation with eigenfunctions and Bessel functions
of \eqref{discovery channel}, that $U$ belongs to the space
$H^1_0(\Omega\times(0,\infty),y^adX)$, which is the completion of $C^\infty_c(\Omega\times[0,\infty))$
under the norm
$$\|U\|^2_{H^1_0(\Omega\times(0,\infty),y^adX)}=\int_\Omega\int_0^\infty y^a\big(U^2+|\nabla U|^2)\,dX.$$
See \cite{Fabes, Fabes-Kenig-Serapioni}, also \cite{Turesson}, for the theory of weighted Sobolev spaces.
It is known (see \cite[Proposition~2.1]{CDDS}) that these weighted Sobolev spaces have the fractional Sobolev spaces $H^s$ defined in \eqref{fractional Sobolev spaces} as trace spaces, that is,
$$\|U(\cdot,0)\|_{H^s}\leq C_{\Omega,a}\|U\|_{H^1_0(\Omega\times(0,\infty),y^adX)}.$$
Therefore, $u(x)=U(x,0)\in H^s$. This and the fact that the norm in $H^1_0(\Omega\times(0,\infty),y^adX)$
is comparable to the natural energy for the extension equation given by \eqref{jota} show that
$\H^s=H^s$, for all $0<s<1$, as we already mentioned at the end of Remark \ref{rem:H fancy en Hs}.
We sumarize all these considerations in the following result. See \cite{Stinga, Stinga-Torrea}.

\begin{thm}[Extension problem]\label{thm:extension}
 Let $u\in\mathcal{H}^s$. There exists a unique weak solution $U\in H^1_0(\Omega\times(0,\infty),y^adX)$
 to the extension problem \eqref{dos estrellitas}, where $B(x)$ and $a$ are as in \eqref{B}.
Moreover, $U$ is given by \eqref{formula con el semigrupo}, which can be also written
as \eqref{Poisson formula}, and
it satisfies \eqref{Neumann condition}. More precisely, for each $\varphi\in H^1_0(\Omega\times(0,\infty),y^adX)$,
$$\int_\Omega\int_0^\infty y^aB(x)\nabla U\nabla\varphi\,dX=c_s\int_\Omega L^su(x)\varphi(x,0)\,dx.$$
In particular, $U$ is the unique minimizer of the energy functional
\begin{equation}\label{jota}
\mathcal{J}(U)=\int_\Omega\int_0^\infty y^aB(x)\nabla U\nabla U\,dX,
\end{equation}
over the set $\mathcal{U}=\{U\in H^1_0(\Omega\times(0,\infty),y^adX):U(x,0)=u(x)\}$, and for the
minimizer $U$ we have the identity
$$\int_\Omega\int_0^\infty y^aB(x)\nabla U\nabla U\,dX=\|L^{s/2}u\|_{L^2(\Omega)}^2
=\|u\|_{\mathcal{H}^s}^2,$$
and the inequality
$$\int_\Omega\int_0^\infty y^a|U|^2\,dX\leq C_{\Omega,s}\|u\|_{L^2(\Omega)}^2.$$
\end{thm}

%%%%%%%%%%%%%%%%%%%%%%%%%%%%%%%%%%%%%%%%%%%%%%%%%%%%%%
\subsection{Scaling}
%%%%%%%%%%%%%%%%%%%%%%%%%%%%%%%%%%%%%%%%%%%%%%%%%%%%%%

For $u\in\H^s$ and $\lambda>0$, let
$$A_\lambda(x):=A(\lambda x),\quad u_\lambda(x):=u(\lambda x),$$
and call $L_\lambda$ the operator with coefficients $A_\lambda$ in $\Omega_{\lambda}
:=\{x: x/\lambda\in\Omega\}$. Then
$$(L_\lambda^su_\lambda)(x)=\lambda^{2s}(L^su)(\lambda x),\quad\hbox{in}~\Omega_{\lambda}.$$
In particular, if $L$ has constant coefficients (as in the case of the Dirichlet Laplacian)
then $L^su_\lambda(x)=\lambda^{2s}L^su(\lambda x)$, in $\Omega_{\lambda}$.
We present two different proofs.

\begin{proof}[Proof using the semigroup]
Let $v(x,t)=e^{-tL}u(x)$. Since $L$ is a linear second order divergence form elliptic operator,
it follows that $v$ satisfies the usual parabolic scaling. This immediately implies that the heat 
semigroup for $L_\lambda$ is given by
$$e^{-tL_\lambda}u_\lambda(x)=v(\lambda x,\lambda^2t),\quad x\in \Omega_{\lambda},~t>0.$$
Now, by Lemma \ref{lem:semi} and the change of variables $r=\lambda^2t$, we see that the following
identities hold in the weak sense:
\begin{align*}
 (L_\lambda^su_\lambda)(x) &= \frac{1}{\Gamma(-s)}\int_0^\infty
 \big(e^{-tL_\lambda}u_\lambda(x)-u_\lambda(x)\big)\,\frac{dt}{t^{1+s}} \\
  &= \frac{1}{\Gamma(-s)}\int_0^\infty\big(v(\lambda x,\lambda^{2}t)-u(\lambda x)\big)\,\frac{dt}{t^{1+s}} \\
  &= \frac{\lambda^{2s}}{\Gamma(-s)}\int_0^\infty\big(v(\lambda x,r)-u(\lambda x)\big)\,\frac{dr}{r^{1+s}}
  =\lambda^{2s}(L^su)(\lambda x).
\end{align*}
\end{proof}

\begin{proof}[Proof using the extension problem]
Let $U$ be the solution to the extension problem \eqref{dos estrellitas}.
Consider $U_\lambda(x,y)=U(\lambda x,\lambda y)$. This function is defined for $x$ in $\Omega_{\lambda}$ and $y>0$.
Then, by using the weak formulation of the extension problem 
it is easy to check that
\begin{equation}\label{extension u tilde}
\begin{cases}
   \dive(y^aB_\lambda(x)\nabla U_\lambda)=0,&\hbox{for}~x\in\Omega_{\lambda},~y>0,\\
   U_\lambda(x,y)=0,& \hbox{for}~x\in\partial\Omega_\lambda,~y\geq0,\\
   U_\lambda(x,0)=u_\lambda(x),&\hbox{for}~x\in\Omega_{\lambda},
\end{cases}
\end{equation}
where $B_\lambda(x)=B(\lambda x)$, $x\in\Omega_{\lambda}$.
Since \eqref{extension u tilde} is the extension problem for $u_\lambda$ and the operator
$L_\lambda^s$,
$$-y^a\partial_yU_\lambda(x,y)\Big|_{y=0}=c_s(L_\lambda^su_\lambda)(x),$$
in $L^2(\Omega)$.
To conclude notice that
$$-y^a\partial_yU_\lambda(x,y)\Big|_{y=0} 
=-\lambda^{2s}(\lambda y)^aU_y(\lambda x,\lambda y)\Big|_{(\lambda y)=0}=c_s\lambda^{2s}(L^su)(\lambda x).$$
\end{proof}

%%%%%%%%%%%%%%%%%%%%%%%%%%%%%%%%%%%%%%%%%%%%%%%%%%%%%%
\subsection{Fundamental solution}
%%%%%%%%%%%%%%%%%%%%%%%%%%%%%%%%%%%%%%%%%%%%%%%%%%%%%%

The fundamental solution $G_s(x,z)=G_s^z(x)$ (Green function)
of $L^s$ with pole at $z\in\Omega$ is defined as the weak solution to
\begin{equation*}
\begin{cases}
  L^sG^z=\delta_z,~G\geq0,& \hbox{in}~\Omega,\\
  G^z=0,& \hbox{on}~\partial\Omega.
\end{cases}
\end{equation*}
Then $G_{s}(x,z)$ is the distributional kernel of $L^{-s}$, namely,
\begin{equation}\label{fundamental series}
G_{s}(x,z)=\sum_{k=0}^\infty\frac{1}{\lambda_k^s}\phi_k(z)\phi_k(x)=G_s(z,x),\quad\hbox{in}~\H^{-s}.
\end{equation}
Indeed, for any $\psi=\sum_{k=0}^\infty d_k\phi_k\in\H^s$, by the symmetry of $L^s$,
$$\langle L_x^sG^z_{s},\psi\rangle = \langle G^z_{s},L_x^s\psi\rangle
  =\sum_{k=0}^\infty \frac{1}{\lambda_k^s}\phi_k(z)\lambda_k^sd_k =\psi(z).$$
  
The following result is in the spirit of Littman--Stampacchia--Weinberger \cite{L-S-W}.
The proof is done by using the extension problem.

\begin{thm}[Littman--Stampacchia--Weinberger-type estimate]\label{thm:LSW}
Fix the ellipticity constants $0<\Lambda_1\leq\Lambda_2$.
Then the fundamental solutions $G_s$ of any of the operators $L^s$ that have ellipticity constants
between $\Lambda_1$ and $\Lambda_2$
satisfy the following property. For any compact subset $\mathcal{K}\subset\Omega$
there exist positive constants $C_1,C_2$, depending only on $\mathcal{K}$, $\Omega$,
$\Lambda_1$, $\Lambda_2$ and $s$
such that, when $n>2s$,
$$\frac{C_1}{|x-z|^{n-2s}}\leq G_s(x,z)\leq\frac{C_2}{|x-z|^{n-2s}},\quad x,z\in\mathcal{K},~x\neq z.$$
In the case $n=2s$ we must replace $|x-z|^{n-2s}$ by $-\ln|x-z|$.
\end{thm}

\begin{proof}
We do the proof for $G_{s}(x,0)$, the fundamental solution of $L^s$ with pole at the origin,
and for $\Omega=Q_1$, the cube with center at the origin an side length $1$. We take $\mathcal{K}$
to be $Q_{1/4}\subset Q_1$.
Let $U$ be the solution to the following extension problem
\begin{equation*}
\begin{cases}
  \dive(y^{a}B(x)\nabla U)=0,& \hbox{in}~Q_1\times(0,\infty),\\
  U(x,y)=0,& \hbox{on}~\partial Q_1\times[0,\infty),\\
  -\lim_{y\to0}y^{a}U_y(x,y)=c_s\delta_{(0,0)},& \hbox{on}~Q_1.  
\end{cases}
\end{equation*}
Then, $U(x,0)=G_{s}(x,0)$, see \cite{Stinga, Stinga-Torrea}. Because of the Neumann boundary condition
at $y=0$, it is easy to check (see the technique in \cite{Caffarelli-Silvestre} or \cite{Stinga-Zhang})
that the even reflection $\widetilde{U}(x,y)=U(x,|y|)$, $x\in Q_1$, $y\in\R$,
is a weak solution to the equation
$$\begin{cases}
  \dive(|y|^{a}B(x)\nabla\widetilde{U})=c_s\delta_{(0,0)},& \hbox{in}~Q_1\times\R,\\
  \widetilde{U}=0,& \hbox{on}~\partial Q_1\times\R.
\end{cases}$$
This is a degenerate elliptic equation with $A_2$ weight $\omega(x,y)=|y|^{a}$ in $\R^{n+1}$.
By the result of Fabes, Jerison and Kenig \cite{Fabes-Jerison-Kenig} (see also Fabes \cite{Fabes}), the Green function
$\widetilde{U}(x,y)$ is comparable in $Q_{1/4}$ to the quantity
$$\int_{|x-z|}^1\frac{s}{s^{n+a+1}}\,ds\sim c_{n,s}
\begin{cases}
|x-z|^{-(n-2s)},&\hbox{if}~n>2s,\\
\ln\frac{1}{|x-z|},&\hbox{if}~n=2s.
\end{cases}$$
\end{proof}

One can also apply the language of semigroups to study the fundamental
solution of $L^s$ as explained in \cite{Stinga, Stinga-Torrea}.
If we use the numerical formula
 $$\lambda^{-s}=\frac{1}{\Gamma(s)}\int_0^\infty e^{-t\lambda}\,\frac{dt}{t^{1-s}},\quad \lambda,s>0,$$
 in \eqref{fundamental series}, we see that the fundamental solution can be written as
 \begin{equation}\label{eq:fundamental con semigrupo}
 G_s(x,z)=\frac{1}{\Gamma(s)}\int_0^\infty W_t(x,z)\,\frac{dt}{t^{1-s}},
\quad\hbox{in}~\mathcal{H}^{-s},
\end{equation}
where $W_t(x,z)$ is the heat kernel for $L$, see \eqref{eq:heat kernel}.
 The next estimates should be compared with those of Theorem \ref{thm:estimaciones Ks}.

\begin{thm}[Estimates for $G_s(x,z)$]\label{thm:estimaciones Gs}
Let $G_s(x,z)\geq0$ be the fundamental solution of $L^s$.
\begin{enumerate}[$(1)$]
\item If the coefficients $A(x)$ are bounded and measurable then
$$G_s(x,z)\leq\frac{c_{n,s}}{|x-z|^{n-2s}},\quad x,z\in\Omega,~x\neq z.$$
\item If the coefficients $A(x)$ are bounded and measurable in $\Omega=\R^n$ then
$$G_s(x,z)\sim\frac{c_{n,s}}{|x-z|^{n-2s}},\quad x,z\in\R^n,~x\neq z.$$
\item If the coefficients $A(x)$ are H\"older continuous in $\Omega$ with exponent $\alpha\in(0,1)$
 then there exist positive constants $c$
and $\eta\leq 1\leq\rho$
depending only on $n$, $\alpha$, $\Omega$ and ellipticity, with $c$ depending also on $s$, such that
$$c^{-1}\min\bigg(1,\frac{\phi_0(x)\phi_0(z)}{|x-z|^{2\eta}}\bigg)\frac{1}{|x-z|^{n-2s}}
\leq G_s(x,z)\leq c\min\bigg(1,\frac{\phi_0(x)\phi_0(z)}{|x-z|^{2\rho}}\bigg)\frac{1}{|x-z|^{n-2s}},$$
for $x,z\in\Omega$, $x\neq z$,
where $\lambda_0$ and $\phi_0$ are the first eigenvalue and 
the first eigenfunction of $L$.
\item Under the hypothesis of $(3)$, if in addition
$\Omega$ is a $C^{1,\gamma}$ domain for some $0<\gamma<1$, then
the estimate in $(3)$ is true for $\eta=\rho=1$ and $c$ depending also on $\gamma$.
In particular, the estimate holds when $G^s$
is the fundamental solution of the fractional Dirichlet Laplacian $(-\Delta_D)^s$ in a $C^{1,\gamma}$ domain.
\end{enumerate}
\end{thm}

\begin{proof}
The proof is parallel to that of Theorem \ref{thm:estimaciones Ks} by using the
heat kernel estimates given there and then integrating in $t$ in
the identity \eqref{eq:fundamental con semigrupo} via the change of variables $r=|x-z|^2/t$.
\end{proof}

%%%%%%%%%%%%%%%%%%%%%%%%%%%%%%%%%%%%%%%%%%%%%%%%%%%%%%
\subsection{Harnack inequality and De Giorgi--Nash--Moser theory}
%%%%%%%%%%%%%%%%%%%%%%%%%%%%%%%%%%%%%%%%%%%%%%%%%%%%%%

Let $u\in\mathcal{H}^s$, $u\geq0$ in $\Omega$ such that $L^su=0$ in some ball $B\subset\subset\Omega$.
Then there exists a constant $C$ depending on $B$, $\Omega$, $n$ and $s$ such that
$$\sup_{\frac{1}{2}B}u\leq C\inf_{\frac{1}{2}B}u.$$
Moreover, $u$ is $\alpha$--H\"older continuous in $\tfrac{1}{2}B$, for some exponent $0<\alpha<1$.
This result can be proved by using the extension problem of \cite{Stinga-Torrea} as stated
in Theorem \ref{thm:extension}. For details see \cite{Stinga-Zhang}.

%%%%%%%%%%%%%%%%%%%%%%%%%%%%%%%%%%%%%%%%%%%%%%%%%%%%%%
\section{Caccioppoli estimate, approximation, regularity of harmonic functions
 and a trace inequality}\label{section:3}
%%%%%%%%%%%%%%%%%%%%%%%%%%%%%%%%%%%%%%%%%%%%%%%%%%%%%%

In this section we consider solutions $U\in H^1(B_1^*,y^adX)$ to
\begin{equation}\label{perturbed extension}
\begin{cases}
  \dive(y^{a}B(x)\nabla U)=\dive(y^aF),& \hbox{in}~B_1^*,\\
  -y^{a}U_y\big|_{y=0}=f,& \hbox{on}~B_1,
\end{cases}
\end{equation}
where $B(x)$ is given by \eqref{B} and $F=(F_1,\ldots,F_{n+1})$ is a vector field on $B_1^*$ such that
\begin{equation}\label{F}
F_i(x)\in L^2(B_1^*,y^adX),~ i=1,\ldots,n,\quad\hbox{and}\quad F_{n+1}=0.
\end{equation}

\begin{defn}\label{def:weak}
A function $U\in H^1(B_1^*,y^adX)$ is a weak solution to \eqref{perturbed extension} if
$$\int_{B_1^*}y^{a}B(x)\nabla U\nabla\psi\,dX=\int_{B_1^*}y^{a}F\nabla\psi \,dX
+\int_{B_1}\psi(x,0)f(x)\,dx,$$
for every $\psi\in H^1(B_1^*,y^{a}dX)$ such that $\psi=0$ on $\partial B_1^*\setminus
 (\overline{B_1}\times\{0\})$.
 \end{defn}
 
 By a change of coordinates we can always assume that
 $$B(0)=I.$$

\begin{lem}[Caccioppoli inequality]\label{Lem:Caccioppoli}
 Let $U$ be a weak solution to \eqref{perturbed extension} in the sense 
 of Definition \ref{def:weak} with $F$ as in \eqref{F}.
 Then, for every $\eta\in C^\infty(\overline{B_1^*})$ that vanishes on $\partial B_1^*\setminus
 (\overline{B_1}\times\{0\})$,
 $$\int_{B_1^*}y^{a}\eta^2|\nabla U|^2\,dX\leq 
 C\left(\int_{B_1^*}y^{a}\left(|\nabla\eta|^2U^2+|F|^2\eta^2\right)\,dX
 +\int_{B_1}(\eta(x,0))^2|U(x,0)||f(x)|\,dx\right),$$
 where $C=C(\lambda,\Lambda)$.
\end{lem}

\begin{proof}
 Take $\psi=\eta^2U\in H^1(B_1^*,y^adX)$ as a test function. Then
 \begin{align*}
 \int y^{a} B(x)\eta^2\nabla U\nabla U\,dX &= -2\int y^{a}\eta U B(x)\nabla U\nabla\eta\,dX
 +\int(\eta(x,0))^2U(x,0)f(x)\,dx \\
 &\quad +\int y^a\eta^2F\nabla U\,dX+2\int y^aU\eta F\nabla\eta\,dX.
 \end{align*}
Using the ellipticity and the Cauchy inequality with $\varepsilon>0$,
\begin{align*}
 \lambda\int y^{a}\eta^2|\nabla U|^2\,dX &\leq \frac{\Lambda}{2\varepsilon}
 \int y^{a}U^2|\nabla\eta|^2\,dX+2\Lambda\varepsilon\int y^{a}\eta^2|\nabla U|^2\,dX
 +\int\eta^2|U||f|\,dx \\
 &\quad +\frac{1}{4\varepsilon}\int y^a\eta^2|F|^2\,dX
 +\varepsilon\int y^a\eta^2|\nabla U|^2\,dX\\
 &\quad +\frac{1}{2\varepsilon}\int y^a\eta^2|F|^2\,dX
 +2\varepsilon\int y^aU^2|\nabla\eta|^2\,dX.
 \end{align*}
 The inequality follows by choosing $\varepsilon$ such that $(2\Lambda+1)\varepsilon/\lambda<1/2$.
\end{proof}

By a compactness argument we get the following consequence
of the Caccioppoli inequality.

\begin{cor}[Approximation lemma]\label{Cor:approximation extension}
Let $U$ be a weak solution to \eqref{perturbed extension} in the sense
of Definition \ref{def:weak} with $F$ as in \eqref{F}.
 Suppose that $U$ is normalized so that
$$ \int_{B_1}U(x,0)^2\,dx+\int_{B_1^*}y^{a}U^2\,dX\leq1.$$
 Then for every $\varepsilon>0$ there exists $\delta=\delta(\varepsilon)>0$
 such that if
 $$\int_{B_1}f^2\,dx+\int_{B_1^\ast}y^a|F|^2\,dX
+\int_{B_1} |A(x)-I|^2\,dx<\delta^2,$$
 then there exists a solution $W$ to
\begin{equation}\label{Neumann armonica perturbed}
\begin{cases}
  \dive(y^{a}\nabla W)=0,& \hbox{in}~B_{3/4}^*,\\
  -y^{a}W_y\big|_{y=0}=0,& \hbox{on}~B_{3/4},
\end{cases}\end{equation}
such that
$$\int_{B_{3/4}^*}|U-W|^2y^{a}\,dX<\varepsilon^2.$$
\end{cor}

\begin{proof}
We prove it by contradiction. Suppose that there exist $\varepsilon_0>0$,
coefficients $A_k$, weak solutions $U_k$ in $B_1^*$, Neumann type data $f_k$
and right hand sides $F^k$, such that
$$\int_{B_1}U_k^2\,dx+\int_{B_1^*}y^{a}U_k^2\,dX\leq1,$$
and
$$\int_{B_1}f_k^2\,dx+\int_{B_1^*}y^a|F^k|^2\,dX
+\int_{B_1} |A_k(x)-I|^2\,dx<\frac{1}{k^2},$$
so that for any solution $W$ to \eqref{Neumann armonica perturbed},
\begin{equation}\label{a contradecir}
\int_{B_{3/4}^*}|U_k-W|^2y^{a}\,dX\geq\varepsilon_0^2,
\end{equation}
for all $k\geq1$.
Let $\eta$ be a test function which is equal to $1$ in $\overline{B_{3/4}^*}$
and vanishes outside $B_1^*$. Then the Caccioppoli estimate and the hypotheses imply that
$$\int_{B_{3/4}^*}y^{a}|\nabla U_k|^2\,dX\leq C,\quad\hbox{for all}~k.$$
Therefore, $\{U_k\}_{k\geq1}$ is a bounded sequence in $H^1(B_{3/4}^*,y^{a}dX)$.
Hence, by compactness of the Sobolev embedding, there exists a subsequence,
that we still denote by $U_k$, and a function $U_\infty$ such that
$$\label{subsequence}
 \begin{cases}
  U_k\to U_\infty,&\hbox{weakly in}~H^1(B_{3/4}^*,y^{a}dX),~\hbox{and}\\
  U_k\to U_\infty,&\hbox{strongly in}~L^2(B_{3/4}^*,y^{a}dX).
 \end{cases}$$
We show now that $U_\infty$ is a solution to \eqref{Neumann armonica perturbed}, which will give
us a contradiction. Indeed, for any suitable test function $\psi$,
$$\int_{B_{3/4}^*}y^{a}B_k(x)\nabla U_k\nabla\psi\,dX=\int_{B_{3/4}^*}y^{a}F^k\nabla\psi\,dX+
\int_{B_{3/4}}\psi(x,0)f_k(x)\,dx.$$
By taking the limit as $k\to\infty$ along the subsequence found above we get
$$\int_{B_{3/4}^*}y^{a}\nabla U_\infty\nabla\psi\,dX=0.$$
This contradicts \eqref{a contradecir} for $W=U_\infty$ and $k$ sufficiently large.
\end{proof}

\begin{rem}[Approximation up to the boundary]\label{remark1}
The following observation will be useful when studying the boundary regularity
 for the fractional problem with Dirichlet boundary condition.
We say that $U\in H^1((B_1^+)^*,y^adX)$ is a weak solution to the half ball problem
\begin{equation}\label{media bola}
\begin{cases}
  \dive(y^{a}B(x)\nabla U)=\dive(y^aF),& \hbox{in}~(B_1^+)^*,\\
  -y^aU_y\big|_{y=0}=f,& \hbox{on}~B_1^+,\\
  U=0,&\hbox{on}~B_1\cap\{x_n=0\},
\end{cases}
\end{equation}
if $U$ satisfies the identity in Definition \ref{def:weak} with $B_1$ replaced by $B_1^+$.
The test functions vanish on $\partial(B_1^+)^*\setminus(\overline{B_1^+}\times\{0\})$.
With this definition then it is clear that the 
Caccioppoli inequality of Lemma \ref{Lem:Caccioppoli} holds for solutions $U$ of \eqref{media bola}
with $B_1^+$ in place of $B_1$ in the statement.
This allows to prove an approximation lemma parallel to Corollary \ref{Cor:approximation extension}.
Namely, given $\varepsilon>0$, there exists $\delta=\delta(\varepsilon)$ such that if
$U$ is a solution of \eqref{media bola} that satisfies the hypotheses of Corollary \ref{Cor:approximation extension}
with $B_1^+$ in place of $B_1$, then there exists a solution $\mathcal{W}$ to
\begin{equation}\label{label}
\begin{cases}
  \dive(y^{a}\nabla \mathcal{W})=0,& \hbox{in}~(B_{3/4}^+)^*,\\
  -y^{a}\mathcal{W}_y\big|_{y=0}=0,& \hbox{on}~B_{3/4}^+,\\
  \mathcal{W}=0,&\hbox{on}~B_{3/4}\cap\{x_n=0\},
\end{cases}
\end{equation}
such that
$$\int_{(B_{3/4}^+)^*}y^a|V-\mathcal{W}|^2\,dX<\varepsilon^2.$$
\end{rem}

Since we will apply Corollary \ref{Cor:approximation extension}, we need to understand
the regularity of solutions to \eqref{Neumann armonica perturbed}.

\begin{prop}\label{Prop:armonicas}
Let $W\in H^1(B_1^*,y^adX)$ be a weak solution to
$$\begin{cases}
  \dive(y^{a}\nabla W)=0,& \hbox{in}~B_1^*,\\
  -y^{a}W_y\big|_{y=0}=0,& \hbox{on}~B_1.
\end{cases}$$
\begin{enumerate}[$(1)$]
	\item For each integer $k\geq0$ and each $B_r(x_0)\subset B_1$,
 $$\sup_{B_{r/2}(x_0)\times[0,r/2)}|D^k_xW|\leq\frac{C}{r^k}\osc_{B_r(x_0)\times[0,r)}W,$$
 where $C$ depends only on  $n$, $k$ and $s$.
 	\item For each $B_r(x_0)\subset B_1$,
	$$\max_{B_{r/2}(x_0)\times[0,r/2)}|W|\leq M\left(\frac{1}{r^{n+1+a}}\int_{B_r(x_0)^*}y^a|W|^2\,dX\right)^{1/2},$$
 where $M$ depends only on $n$ and $s$.
 	\item We have
	$$\sup_{x\in B_{1/2}}|W_y(x,y)|\leq Cy,$$
	where $C$ depends only on $n$ and $s$.
\end{enumerate}
\end{prop}

\begin{proof}
It can be seen that $\widetilde{W}(x,y):=W(x,|y|)$, $y\in(-1,1)$, is a weak solution to
$$\dive(|y|^a\nabla W)=0,\quad\hbox{in}~B_1\times(-1,1),$$
see \cite[Lemma~4.1]{Caffarelli-Silvestre}.
 Now $(1)$ is contained in \cite[Corollary~2.5]{Caffarelli-Salsa-Silvestre},
 and $(2)$ follows from \cite[Corollary~2.3.4]{Fabes-Kenig-Serapioni}.
 Finally $(3)$ is due to the fact that the Neumann type condition that $W$ satisfies
 has a zero right hand side. Indeed, this follows
 from the proof of Lemma 4.2 in \cite[p.~1254]{Caffarelli-Silvestre}.
\end{proof}

\begin{rem}[Regularity up to the boundary -- Dirichlet]\label{remark2}
For a solution $\mathcal{W}$ to \eqref{label}
we can perform an odd reflection in the $x_n$ variable, that we
call $\mathcal{W}_{o}$, which satisfies
$$\begin{cases}
  \dive(y^{a}\nabla \mathcal{W}_{o})=0,& \hbox{in}~B_{3/4}^*,\\
  -y^{a}(\mathcal{W}_{o})_y\big|_{y=0}=0,& \hbox{on}~B_{3/4}.
\end{cases}$$
Therefore $\mathcal{W}$ is smooth in $\overline{B_{1/2}^+}$.
Also, $(\mathcal{W}_{o})_y$ grows like $y$ near $y=0$, $x_n=0$. Hence $\mathcal{W}$
satisfies the same estimates
as those for harmonic functions contained in Proposition \ref{Prop:armonicas}
up to the boundary $B_{1/2}\cap\{x_n=0\}$.
\end{rem}

In the proof of the regularity estimates we will need to use the following trace inequality
on balls with explicit dependence of the constant in terms of the radius.

\begin{lem}[Trace inequality in balls]\label{lem:inequality}
There exists a constant $C>0$ depending only $n$ and $s$ such that
\begin{equation}\label{eq:trace}
r^{1-s}\|U(\cdot,0)\|_{L^2(B_r)}\leq C\|U\|_{H^1(B_r^*,y^adX)},
\end{equation}
for all $U\in H^1(B_1^*,y^adX)$ and for any $0<r\leq1$. The inequality
is also true if we replace $B_r$ by $B_r^+$.
\end{lem}

\begin{proof}
It is enough to consider $r=1$. For if \eqref{eq:trace} is true in this case then for the general case
we need to take the rescaled function $V(x,y)=U(rx,ry)$. 
Recall that there exists a linear extension operator $E:H^1(B_1^*,|y|^adX)\to
H^1(\R^{n+1},|y|^adX)$, such that $EU=U$ in $B_1^*$ and
\begin{equation}\label{EU2}
\|EU\|_{H^1(\R^{n+1},|y|^adX)}\leq C_0\|U\|_{H^1(B_1^*,|y|^adX)},
\end{equation}
where $C_0$ depends only on $n$ and $s$. Also, $EU$ has compact support.
See for example \cite[Chapter~2,~Theorem~2.1.13]{Turesson}. Moreover, the following trace inequality
of  Lions \cite[Paragraph~5]{Lions}
$$\|F(\cdot,0)\|^2_{H^s(\R^n)}=\|F(\cdot,0)\|_{L^2(\R^n)}^2+[F(\cdot,0)]^2_{H^s(\R^n)}\leq c^2
\|F\|_{H^1(\R^{n+1}_+,y^adX)}^2,$$
holds for any $F\in H^1(\R^{n+1}_+,y^adX)$, with a constant $c$ depending only on $n$ and $s$.
Using this trace inequality with $F=EU$ and \eqref{EU2} we get
\begin{align*}
	\|U(\cdot,0)\|_{L^2(B_1)} &= \|(EU)(\cdot,0)\|_{L^2(B_1)} \leq \|(EU)(\cdot,0)\|_{H^s(\R^n)} \\
	&\leq c\|EU\|_{H^1(\R^{n+1}_+,y^adX)} \leq cC_0\|U\|_{H^1(B_1^*,|y|^adX)}.
\end{align*}
\end{proof}

%%%%%%%%%%%%%%%%%%%%%%%%%%%%%%%%%%%%%%%%%%%%%%%%%%%%%%
\section{Interior regularity}\label{section:4}
%%%%%%%%%%%%%%%%%%%%%%%%%%%%%%%%%%%%%%%%%%%%%%%%%%%%%%

Theorems \ref{thm:interior Calpha} and \ref{thm:interior Lp} are in fact corollaries of 
the more general results that we state and prove in this section.

We say that a function $f:B_1\to\R$ is in $L^{2,\alpha}(0)$, for $0\leq\alpha<1$, whenever
$$[f]_{L^{2,\alpha}(0)}^2:=\sup_{0<r\leq1}\frac{1}{r^{n+2\alpha}}\int_{B_r}|f(x)-f(0)|^2\,dx<\infty,$$
where $f(0)$ is defined as $\displaystyle f(0):=\lim_{r\to0}\frac{1}{|B_r|}\int_{B_r}f(x)\,dx$.
It is clear that if $f$ is H\"older continuous of order $0<\alpha<1$ at $0$ then $f\in L^{2,\alpha}(0)$.
If the condition above holds uniformly in balls centered at points
close to the origin, then $f$ is $\alpha$--H\"older continuous
around the origin, see \cite{Campanato}.

\begin{thm}\label{thm:interior L2alpha}
Let $u$ be a solution to \eqref{uma}.
Assume that $\Omega$ is a bounded Lipschitz domain containing the ball $B_1$ and let $f\in L^{2,\alpha}(0)$,
for some $0<\alpha<1$.
 \begin{enumerate}[$(1)$]
  \item Suppose that $0<\alpha+2s<1$. There exist $0<\delta<1$, depending only on $n$, ellipticity, $\alpha$ and $s$,
  and a constant $C_0>0$ such that if
   $$\sup_{0<r\leq1}\frac{1}{r^n}\int_{B_r}|A(x)-A(0)|^2\,dx<\delta^2,$$
   then there exists a constant $c$ such that
  $$\frac{1}{r^n}\int_{B_r}|u(x)-c|^2\,dx\leq C_1r^{2(\alpha+2s)},\quad\hbox{for all}~r>0~
  \hbox{sufficiently small},$$
  where $C_1+|c|\leq C_0\big(\|u\|_{L^2(\Omega)}+[u]_{H^s(\Omega)}+[f]_{L^{2,\alpha}(0)}+|f(0)|\big)$.
   \item Suppose that $1<\alpha+2s<2$. There exist $0<\delta<1$, depending only on $n$, ellipticity, $\alpha$ and $s$,
    and a constant $C_0>0$ such that if
   $$\sup_{0<r\leq1}\frac{1}{r^{n+2(\alpha+2s-1)}}\int_{B_r}|A(x)-A(0)|^2\,dx<\delta^2,$$
   then there exists a linear function $\ell(x)=\mathcal{A}+\mathcal{B}\cdot x$ such that
     $$\frac{1}{r^n}\int_{B_r}|u(x)-\ell(x)|^2\,dx\leq C_1r^{2(\alpha+2s)},\quad\hbox{for all}~r>0~
     \hbox{sufficiently small},$$
  where $C_1+|\mathcal{A}|+|\mathcal{B}|
  \leq C_0\big(\|u\|_{L^2(\Omega)}+[u]_{H^s(\Omega)}+[f]_{L^{2,\alpha}(0)}+|f(0)|\big)$.
 \end{enumerate}
 The constants $C_0$ above depend only on $[A]_{L^{2,0}(0)}$ (resp. $[A]_{L^{2,\alpha+2s-1}(0)}$),
  ellipticity, $n$, $\alpha$ and $s$.
\end{thm}

It is clear then that Theorem \ref{thm:interior Calpha} for the Dirichlet case
is a direct consequence of Theorem
\ref{thm:interior L2alpha} after a dilation of the variables if necessary. Indeed,
the conditions on $f$ and on the coefficients hold everywhere in $\Omega$
and therefore the estimate for $u$ can be obtained around every interior point.

We say that a function $f:B_1\to\R$ is in $L^{2,-2s+\alpha}(0)$, $0<\alpha<1$, whenever
$$[f]^2_{L^{2,-2s+\alpha}(0)}:=\sup_{0<r\leq1}\frac{1}{r^{n+2(-2s+\alpha)}}\int_{B_r}|f(x)|^2\,dx<\infty.$$
Also, $f$ is in $L^{2,-2s+\alpha+1}(0)$, $0<\alpha<1$, whenever
$$[f]^2_{L^{2,-2s+\alpha+1}(0)}:=\sup_{0<r\leq1}\frac{1}{r^{n+2(-2s+\alpha+1)}}\int_{B_r}|f(x)|^2\,dx<\infty,$$
We have the following consequences of H\"older's inequality.
\begin{itemize}
	\item If $f\in L^p(B_1)$, for $n/(2s)<p<n/(2s-1)^+$, then $f\in L^{2,-2s+\alpha}(0)$ and
	$$[f]_{L^{2,-2s+\alpha}(0)}\leq \|f\|_{L^p(B_1)},$$
	for $\alpha=2s-n/p$.
	\item If $s>1/2$ and $f\in L^p(B_1)$, for $p>n/(2s-1)$, then $f\in L^{2,-2s+\alpha+1}(0)$ and
	$$[f]_{L^{2,-2s+\alpha+1}(0)}\leq \|f\|_{L^p(B_1)},$$
	for $\alpha=2s-n/p-1$.
\end{itemize}

\begin{thm}\label{thm:interior Lmenos2alpha}
Let $u$ be a solution to \eqref{uma}. Assume that $\Omega$ is a bounded Lipschitz domain
containing the ball $B_1$ and let $0<\alpha<1$.
\begin{enumerate}[$(1)$]
  \item Suppose that $f\in L^{2,-2s+\alpha}(0)$.
  There exist $0<\delta<1$, depending only on $n$, ellipticity, $\alpha$ and $s$,
   and a constant $C_0>0$ such that if
   $$\sup_{0<r\leq1}\frac{1}{r^n}\int_{B_r}|A(x)-A(0)|^2\,dx<\delta^2,$$
   then there exists a constant $c$ such that
  $$\frac{1}{r^n}\int_{B_r}|u(x)-c|^2\,dx\leq C_1r^{2\alpha},\quad\hbox{for all}~
  r>0~\hbox{sufficiently small},$$
  where $C_1+|c|\leq C_0\big(\|u\|_{L^2(\Omega)}+[u]_{H^s(\Omega)}+[f]_{L^{2,-2s+\alpha}(0)}\big)$.
   \item Suppose that $f\in L^{2,-2s+\alpha+1}(0)$. 
   There exist $0<\delta<1$, depending only on $n$, ellipticity, $\alpha$ and $s$,
    and a constant $C_0>0$ such that if
   $$\sup_{0<r\leq1}\frac{1}{r^{n+2\alpha}}\int_{B_r}|A(x)-A(0)|^2\,dx<\delta^2,$$
   then there exists a linear function $\ell(x)=\mathcal{A}+\mathcal{B}\cdot x$ such that
     $$\frac{1}{r^n}\int_{B_r}|u(x)-\ell(x)|^2\,dx\leq C_1r^{2(1+\alpha)},\quad\hbox{for all}
     ~r>0~\hbox{sufficiently small},$$
  where $C_1+|\mathcal{A}|+|\mathcal{B}|
  \leq C_0\big(\|u\|_{L^2(\Omega)}+[u]_{H^s(\Omega)}+[f]_{L^{2,-2s+\alpha+1}(0)}\big)$.
 \end{enumerate}
The constants $C_0$ above depend only on $[A]_{L^{2,0}(0)}$ (resp. $[A]_{L^{2,\alpha}(0)}$),
ellipticity, $n$, $\alpha$ and $s$.
\end{thm}

In view of the comments above, Theorem \ref{thm:interior Lp} is a direct corollary of
Theorem \ref{thm:interior Lmenos2alpha} after a dilation of the variables if necessary.

The rest of this section is devoted to the proof of Theorems \ref{thm:interior L2alpha}
and \ref{thm:interior Lmenos2alpha}.

%%%%%%%%%%%%%%%%%%%%%%%%%%%%%%%%%%%%%%%%%%%%%%%%%%%%%%
\subsection{Proof of Theorem \ref{thm:interior L2alpha}(1)}\label{subsection:proof interior(1)}
%%%%%%%%%%%%%%%%%%%%%%%%%%%%%%%%%%%%%%%%%%%%%%%%%%%%%%

It is enough to prove the regularity for $u(x)=U(x,0)$, where $U\in H^1(B_1^*,y^adX)$ is
a solution to
\begin{equation}\label{ext Laplacian}
\begin{cases}
  \dive(y^{a}B(x)\nabla U)=\dive(y^aF),& \hbox{in}~B_1^*,\\
  -y^{a}U_y\big|_{y=0}=f,& \hbox{on}~B_1.
\end{cases}
\end{equation}
Here we take $F$ to be a $B_1^*$--valued vector field in an appropriate Morrey space
(see (3) below) such that $F_{n+1}=0$. Theorem \ref{thm:interior L2alpha}(1) then follows
by taking into account Theorem \ref{thm:extension}, where $F\equiv0$.

It is clear that after an orthogonal
change of variables we can assume that $A(0)=I$. We can also assume that $f(0)=0$. For if $f(0)\neq0$ we take
$$\widetilde{U}(x,y)=U(x,y)+\tfrac{1}{1-a}y^{1-a}f(0),$$
that solves \eqref{ext Laplacian} with Neumann data $\widetilde{f}(x)=f(x)-f(0)$
(recall that $B_{n+1,n+1}(x)=1$) and $\widetilde{f}(0)=0$.

Let $\delta>0$. By scaling and by considering
$$\widetilde{U}(x,y)=U(x,y)\left(\int_{B_1}U(x,0)^2dx+\int_{B_1^*}y^aU^2\,dX
+\frac{1}{\delta}\left([f]_{L^{2,\alpha}(0)}+[F]_{\alpha,s}\right)\right)^{-1},$$
we can suppose the following.
 \begin{enumerate}[$(1)$]
   \item ($A$ has small $L^{2,0}(0)$ seminorm)
  $\displaystyle\sup_{0<r\leq1}\frac{1}{r^{n}}\int_{B_r}|A(x)-I|^2\,dx<\delta^2$;
  \item ($f$ has small $L^{2,\alpha}(0)$ seminorm)
  $\displaystyle[f]_{L^{2,\alpha}(0)}^2=\sup_{0<r\leq1}\frac{1}{r^{n+2\alpha}}\int_{B_r}|f|^2\,dx<\delta^2$;
  \item ($F$ has small Morrey seminorm at 0)
  $\displaystyle[F]_{\alpha,s}:
  =\sup_{0<r\leq1}\frac{1}{r^{n+1+a+2(\alpha+2s-1)}}\int_{B_r^*}y^a|F|^2\,dX<\delta^2$;
   \item ($U$ has bounded $L^2$ norms) 
  $\displaystyle \int_{B_1}U(x,0)^2\,dx+\int_{B_1^*}y^aU^2\,dX\leq1$.
  \end{enumerate}
Given $\delta>0$, a solution $U$ to \eqref{ext Laplacian} is called \textit{normalized}
if $f(0)=0$, $A(0)=I$ and $(1)$--$(4)$ above holds.

Now we prove that given $0<\alpha+2s<1$ there exists $0<\delta<1$, depending only on $n$, ellipticity,
$\alpha$ and $s$, such that 
for any normalized solution $U$ to \eqref{ext Laplacian} there exists a constant $c_\infty$
such that
\begin{equation}\label{a probar constant}
\frac{1}{r^n}\int_{B_r}|U(x,0)-c_\infty|^2\,dx\leq C_0r^{2(\alpha+2s)},\quad\hbox{for all}~r>0~\hbox{sufficiently small},
\end{equation}
and $|c_\infty|\leq C_0$, for some constant $C_0$ depending only on $n$, ellipticity, $\alpha$ and $s$.

\begin{lem}\label{bola fija}
 Given $0<\alpha+2s<1$
 there exist $0<\delta,\lambda<1$, a constant $c$ and a universal constant $D>0$ such that
 for any normalized solution $U$ to \eqref{ext Laplacian} we have
 $$\frac{1}{\lambda^n}\int_{B_\lambda}|U(x,0)-c|^2\,dx+\frac{1}{\lambda^{n+1+a}}
 \int_{B_\lambda^*}|U-c|^2y^a\,dX<\lambda^{2(\alpha+2s)},$$
 and $|c|\leq D$.
\end{lem}

\begin{proof}
 Let $0<\varepsilon<1$ to be fixed. Then there exist $0<\delta<1$ and a harmonic function $W$ that
 satisfy Corollary \ref{Cor:approximation extension}. We have
 $$\int_{B_{1/2}^*}|W|^2y^a\,dX\leq 2\int_{B_{1/2}^*}|U-W|^2y^a\,dX
 +2\int_{B_{1/2}^*}U^2y^a\,dX\leq
 2\varepsilon^2+2\leq 4.$$
 Define $c=W(0,0)$. By the estimates on harmonic functions given
 in Proposition \ref{Prop:armonicas}(2),
 there exists a universal constant $D$ such that $|c|\leq D$.
 Moreover, for any $X\in B_{1/4}^*$, by Proposition \ref{Prop:armonicas}(1)-(3),
 \begin{align*}
 |W(X)-c| &\leq |W(x,y)-W(x,0)|+|W(x,0)-W(0,0)| \\
 &\leq |W_y(x,\xi)|y+\|\nabla_xW\|_{L^\infty(B_{1/4})}|x|\leq N(y^2+|x|)\leq N|X|,
 \end{align*}
 for some universal constant $N$. For any $0<\lambda<1/4$,
\begin{equation}\label{arriba}
 \begin{aligned}
  \frac{1}{\lambda^{n+1+a}}&\int_{B_\lambda^*}|U-c|^2y^a\,dX \\
  &\leq\frac{2}{\lambda^{n+1+a}}\int_{B_\lambda^*}|U-W|^2y^a\,dX+\frac{2}{\lambda^{n+1+a}}
  \int_{B_\lambda^*}|W-c|^2y^a\,dX \\
  &\leq \frac{2\varepsilon^2}{\lambda^{n+1+a}}
  +\frac{2N^2}{\lambda^{n+1+a}}\int_{B_\lambda^*}|X|^2y^a\,dX
  \leq \frac{2\varepsilon^2}{\lambda^{n+1+a}}+c_{n,a}\lambda^2.
 \end{aligned}
 \end{equation}

On the other hand, we apply the trace inequality \eqref{eq:trace}
 to $U-c\in H^1(B_1^*,y^adX)$ to get,
for any $0<\lambda<1/8$,
\begin{equation}\label{trace}
\lambda^{1+a} \int_{B_{\lambda}}|U(x,0)-c|^2\,dx\leq C\left(\int_{B_\lambda^*}|U-c|^2y^a\,dX
 +\int_{B_\lambda^*}|\nabla U|^2y^a\,dX\right).
\end{equation}
Next we need to control the gradient in the right hand side of \eqref{trace}. To that end
we use the Caccioppoli inequality 
 in Lemma \ref{Lem:Caccioppoli} that ensures that, since $U-c$ is also a solution
 to the extension equation with the same Neumann type datum $f$,
\begin{align*}
 \int_{B_\lambda^*}&|\nabla U|^2y^a\,dX \leq C\left(\int_{B_{2\lambda}^*}\big(|U-c|^2
 +|F|^2\big)y^a\,dX+
 \int_{B_{2\lambda}}|U(x,0)-c||f(x)|\,dx\right)\\
 &\leq C\int_{B_{2\lambda}^*}|U-c|^2y^a\,dX+\|F\|_{L^2(B_{2\lambda}^*,y^adX)}^2+
 C\big(\|U(\cdot,0)\|_{L^2(B_{2\lambda})}+|c||B_{2\lambda}|^{1/2}\big)\|{f}\|_{L^2(B_{2\lambda})}.
\end{align*}
Then, since we are in the normalization situation, in \eqref{trace} we get
\begin{equation}\label{estimate trace}
 \lambda^{1+a}\int_{B_{\lambda}}|U(x,0)-c|^2\,dx\leq C\int_{B_{2\lambda}^*}|U-c|^2y^a\,dX
 +\delta^2+C(1+|c|)\delta,
\end{equation}
where $C$ depends only on ellipticity, $n$ and $a$.
 Therefore, for any $0<\lambda<1/8$, from \eqref{estimate trace} and \eqref{arriba} we get
 \begin{align*}
 \frac{1}{\lambda^n}\int_{B_{\lambda}}|U(x,0)-c|^2\,dx 
 &\leq \frac{C}{\lambda^{n+1+a}}\int_{B_{2\lambda}^*}|U-c|^2y^a\,dX
 +\frac{C(1+|c|)}{\lambda^{n+1+a}}(\delta+\delta^2) \\
 &\leq \frac{C\varepsilon^2}{\lambda^{n+1+a}}+c_{n,a}\lambda^2
 +\frac{C(1+D)}{\lambda^{n+1+a}}\delta.
 \end{align*}
 Hence, for any $0<\lambda<1/8$, from this
 and \eqref{arriba},
  $$\frac{1}{\lambda^n}\int_{B_\lambda}|U(x,0)-c|^2\,dx+\frac{1}{\lambda^{n+1+a}}
 \int_{B_\lambda^*}|U-c|^2y^a\,dX<\frac{C\varepsilon^2}{\lambda^{n+1+a}}
 +c_{n,a}\lambda^2+\frac{C\delta}{\lambda^{n+1+a}},$$
 where $C$ depends only on ellipticity, $n$ and $a$ and it is universal for any $W$.
 We first take $0<\lambda<1/8$ sufficiently small in such a way that the second term in the right hand side
 above is less than $\frac{1}{3}\lambda^{2(\alpha+2s)}$. Then
 we let $\varepsilon>0$ small enough so that the first term is less than
 $\frac{1}{3}\lambda^{2(\alpha+2s)}$. For this choice of $\varepsilon$ we take $0<\delta<1$
 in the approximation lemma (Corollary \ref{Cor:approximation extension})
 to be so small in such a way that the third term above is
 smaller than $\frac{1}{3}\lambda^{2(\alpha+2s)}$. Hence
 there exist a constant $c$ bounded by a universal constant $D>0$ and $0<\delta<1$
 such that for any normalized solution $U$ and for some fixed $0<\lambda<1/8$,
   $$\frac{1}{\lambda^n}\int_{B_\lambda}|U(x,0)-c|^2\,dx+\frac{1}{\lambda^{n+1+a}}
 \int_{B_\lambda^*}|U-c|^2y^a\,dX<\lambda^{2(\alpha+2s)}.$$
\end{proof}

\begin{lem}\label{lem:claim constant}
In the situation of Lemma \ref{bola fija} there is a sequence of constants $c_k$, $k\geq0$, such that
$$|c_k-c_{k+1}|\leq D\lambda^{k(\alpha+2s)},$$
and
$$\frac{1}{\lambda^{kn}}\int_{B_{\lambda^k}}|U(x,0)-c_k|^2\,dx+
\frac{1}{\lambda^{k(n+1+a)}}\int_{B_{\lambda^k}^*}|U-c_k|^2y^a\,dX<\lambda^{2k(\alpha+2s)}.$$
\end{lem}

Lemma \ref{lem:claim constant} is enough to get \eqref{a probar constant}.
Indeed, let $c_\infty=\lim_{k\to\infty}c_k$, which is well defined
because of the estimate for $c_k$. For any $r<1/8$, take $k\geq0$ such that $\lambda^{k+1}<r\leq\lambda^{k}$.
Then
\begin{equation}\label{cuentita}
\begin{aligned}
 \frac{1}{r^n}\int_{B_r}|U(x,0)-c_\infty|^2\,dx &\leq \frac{2}{r^n}\int_{B_r}|U(x,0)-c_k|^2\,dx
 +\frac{2}{r^n}\int_{B_r}|c_k-c_\infty|^2\,dx \\
 &\leq 2\left(\frac{\lambda^{kn}}{r^n}\right)\frac{1}{\lambda^{kn}}
 \int_{B_{\lambda^k}}|U(x,0)-c_k|^2\,dx+C_nD^2\lambda^{2k(\alpha+2s)}\\
 &\leq C_{n,\lambda,D}\lambda^{2k(\alpha+2s)}\leq C_{n,\lambda,D}r^{2(\alpha+2s)}.
\end{aligned}
\end{equation}

\begin{proof}[Proof of Lemma \ref{lem:claim constant}]
The proof is done by induction.
When $k=0$, we take $c_0=c_1=0$. Then the conclusion is true because $U$ is a normalized solution.
Assume that the claim is true for some $k\geq0$. Consider
$$\widetilde{U}(X)=\frac{U(\lambda^kX)-c_k}{\lambda^{(\alpha+2s)k}},\quad X\in B_1^*,$$
where $\lambda$ is as in Lemma \ref{bola fija}.
By applying the change of variables $X=\lambda^kZ$ in the weak formulation
$$\int_{B_{\lambda^k}^*}y^aB(x)\nabla U\nabla\psi\,dX=
\int_{B_{\lambda^k}^*}y^aF\nabla\psi\,dX+\int_{B_{\lambda^k}}f(x)\psi(x,0)\,dx,$$
we get, for $\widetilde{B}(x):=B(\lambda^k x)$, $\widetilde{\psi}(X)=\psi(\lambda^kX)$,
$\widetilde{f}(x)=\lambda^{-k\alpha}f(\lambda^k x)$, $\widetilde{F}(x)=\lambda^{-k(\alpha+2s-1)}
F(\lambda^kx)$ and
$\widetilde{U}$ as above,
$$\int_{B_1^*}y^a\widetilde{B}(x)\nabla\widetilde{U}\nabla\widetilde{\psi}\,dX
=\int_{B_1^*}y^a\widetilde{F}\nabla\widetilde{\psi}\,dX
+\int_{B_1}\widetilde{f}(x)\widetilde{\psi}(x,0)\,dx.$$
Thus $\widetilde{U}$ is a weak solution to
\begin{equation}\label{lo que sale}
\begin{cases}
\dive(y^a\widetilde{B}(x)\nabla\widetilde{U})=\dive(y^a\widetilde{F}),&\hbox{in}~B_1^*,\\
-y^a\widetilde{U}_y\big|_{y=0}=\widetilde{f},&\hbox{on}~B_1.
\end{cases}
\end{equation}
Notice that $\widetilde{A}(0)=I$, $\widetilde{F}_{n+1}=0$ and $\widetilde{f}(0)=0$.
Moreover, by changing variables back and using the induction hypothesis,
\begin{align*}
\frac{1}{r^n}\int_{B_r}(\widetilde{A}(x)-I)^2\,dx&=\frac{1}{(\lambda^kr)^n}
\int_{B_{\lambda^kr}}(A(x)-I)^2\,dx<\delta^2;\\
\frac{1}{r^{n+2\alpha}}\int_{B_r}|\widetilde{f}|^2\,dx&=\frac{1}{(\lambda^kr)^{n+2\alpha}}
\int_{B_{\lambda^kr}}|f|^2\,dx\leq[f]_{L^{2,\alpha}(0)}^2<\delta^2;\\
\frac{1}{r^{n+1+a+2(\alpha+2s-1)}}\int_{B_r}y^a|\widetilde{F}|^2\,dX &= 
\frac{1}{(\lambda^kr)^{n+1+a+2(\alpha+2s-1)}}
\int_{B_{\lambda^kr}}y^a|F|^2\,dX<\delta^2; \\ 
\int_{B_1}\widetilde{U}(x,0)^2\,dx&=
\frac{1}{\lambda^{kn}}\int_{B_{\lambda^k}}\frac{|U(x,0)-c_k|^2}{\lambda^{2k(\alpha+2s)}}\,dx\leq1;\\
\int_{B_1^*}y^a\widetilde{U}^2\,dX&=\frac{1}{\lambda^{k(n+1+a)}}
\int_{B_{\lambda^k}^*}y^a\frac{|U-c_k|^2}{\lambda^{2k(\alpha+2s)}}\,dX\leq1.
\end{align*}
Therefore $\widetilde{U}$ is a normalized solution to \eqref{lo que sale}.
 Hence we can apply Lemma \ref{bola fija} to $\widetilde{U}$ in order to get
$$\frac{1}{\lambda^n}\int_{B_{\lambda}}|\widetilde{U}(x,0)-c|^2\,dx
+\frac{1}{\lambda^{n+1+a}}\int_{B_{\lambda}^*}|\widetilde{U}-c|^2y^a\,dX<\lambda^{2(\alpha+2s)}.$$ 
Now we change variables back in the definition of $\widetilde{U}$ to obtain
$$\frac{1}{\lambda^{(k+1)n}}\int_{B_{\lambda^{k+1}}}|U(x,0)-c_{k+1}|^2\,dx<\lambda^{2(k+1)(\alpha+2s)},$$
and
$$\frac{1}{\lambda^{(k+1)(n+1+a)}}\int_{B_{\lambda^{k+1}}^*}
|\widetilde{U}-c_{k+1}|^2y^a\,dX<\lambda^{2(k+1)(\alpha+2s)},$$
where
$$c_{k+1}=c_k+\lambda^{k(\alpha+2s)}c.$$
Obviously, we have $|c_{k+1}-c_k|=|c\lambda^{k(\alpha+2s)}|\leq D\lambda^{\alpha+2s}$. This proves the induction step.
\end{proof}

%%%%%%%%%%%%%%%%%%%%%%%%%%%%%%%%%%%%%%%%%%%%%%%%%%%%%%
\subsection{Proof of Theorem \ref{thm:interior L2alpha}(2)}
%%%%%%%%%%%%%%%%%%%%%%%%%%%%%%%%%%%%%%%%%%%%%%%%%%%%%%

As in the previous subsection, it is enough to prove the regularity for $u(x)=U(x,0)$, where
$U\in H^1(B_1^*,y^adX)$ is a solution to
\begin{equation}\label{ext Laplacian V}
\begin{cases}
  \dive(y^{a}B(x)\nabla U)=\dive(y^aF),& \hbox{in}~B_1^*,\\
  -y^{a}U_y\big|_{y=0}=f,& \hbox{on}~B_1.
\end{cases}
\end{equation}
Here $F$ is now a $B_1^*$--valued vector field that belongs to an appropriate Campanato space
(see (3) below) and such that $F(0)=0$.

As before, we can assume that $A(0)=I$ and $f(0)=0$. We can also suppose the following for $\delta>0$.
 \begin{enumerate}[$(1)$]
   \item ($A$ has small $L^{2,\alpha+2s-1}(0)$ seminorm)
   $\displaystyle\sup_{0<r\leq1}\frac{1}{r^{n+2(\alpha+2s-1)}}\int_{B_r}|A(x)-I|^2\,dx<\delta^2$;
  \item ($f$ has small $L^{2,\alpha}(0)$ seminorm)
  $\displaystyle[f]_{L^{2,\alpha}(0)}^2=\sup_{0<r\leq1}\frac{1}{r^{n+2\alpha}}\int_{B_r}|f|^2\,dx<\delta^2$;
  \item ($F$ has small Campanato seminorm at 0)
  $\displaystyle\sup_{0<r\leq1}\frac{1}{r^{n+1+a+2(\alpha+2s-1)}}\int_{B_r^*}y^a|F|^2\,dX<\delta^2$;
   \item ($U$ has bounded $L^2$ norms) 
  $\displaystyle \int_{B_1}U(x,0)^2\,dx+\int_{B_1^*}y^aU^2\,dX\leq1$.
  \end{enumerate}
Observe that assumption $(1)$ on the coefficients $A(x)$ is equivalent to ask for the matrix $B(x)$ that
$$\sup_{0<r\leq1}\frac{1}{r^{n+1+a+2(\alpha+2s-1)}}\int_{B_r^*}y^a|B(x)-I|^2\,dX<\delta^2.$$

Now we prove that given $1<\alpha+2s<2$ there exists $0<\delta<1$, depending only on $n$, ellipticity,
$\alpha$ and $s$, such that 
for any normalized solution $U$ to \eqref{ext Laplacian V} there exists a linear function
$\ell_\infty(x)=A_\infty+B_\infty\cdot x$ such that
\begin{equation}\label{a probar ell L}
\frac{1}{r^n}\int_{B_r}|U(x,0)-\ell_\infty|^2\,dx\leq C_0r^{2(\alpha+2s)},\quad\hbox{for all}~r>0~\hbox{sufficiently small},
\end{equation}
and $|A_\infty|+|B_\infty|\leq C_0$, for some constant $C_0$ depending only on $n$, ellipticity, $\alpha$ and $s$.

\begin{lem}\label{a ver}
Given $1<\alpha+2s<2$ there exist $0<\delta,\lambda<1$, a linear function
 $$\ell(x)=A+B\cdot x,$$
 and a universal constant $D>0$ such that for any normalized solution $U$ to \eqref{ext Laplacian V} we have
 $$\frac{1}{\lambda^n}\int_{B_\lambda}|U(x,0)-\ell(x)|^2\,dx+\frac{1}{\lambda^{n+1+a}}
 \int_{B_\lambda^*}|U-\ell|^2y^a\,dX<\lambda^{2(\alpha+2s)},$$
 and $|A|+|B|\leq D$.
\end{lem}

\begin{proof}
 Let $0<\varepsilon<1$ to be fixed. Then there exist $0<\delta<1$ and a harmonic function $W$ that
 satisfy Corollary \ref{Cor:approximation extension}. As before we have
 $$\int_{B_{1/2}^*}|W|^2y^a\,dX\leq 4.$$
 Define $\ell(x)=W(0,0)+\nabla_xW(0,0)\cdot x=:A+B\cdot x$. By the estimates on
 harmonic functions given in Proposition \ref{Prop:armonicas},
 there exists a universal constant $D$ such that $|A|+|B|\leq D$.
 Moreover, for any $X\in B_{1/4}^*$, by Proposition \ref{Prop:armonicas},
 \begin{align*}
 |W(x,y)-\ell(x)| &= \left|\big(W(x,y)-W(x,0)\big)+\big(W(x,0)-W(0,0)-\nabla_xW(0,0)\cdot x\big)\right| \\
 &\leq |W_y(x,\xi)|y+\tfrac{1}{2}|D_x^2W(\xi,0)||x|^2\\
 &\leq C\xi y+\tfrac{1}{2}\|D_x^2W\|_{L^\infty(B_{1/4}^*)}|x|^2\leq N|X|^2,
 \end{align*}
 for some universal constant $N$. For any $0<\lambda<1/4$,
 \begin{equation}\label{primera}
 \begin{aligned}
  \frac{1}{\lambda^{n+1+a}}\int_{B_\lambda^*}&|U-\ell|^2y^a\,dX \\
  &\leq \frac{2}{\lambda^{n+1+a}}\int_{B_{\lambda}^*}|U-W|^2y^a\,dX+
  \frac{2}{\lambda^{n+1+a}}\int_{B_{\lambda}^*}|W-\ell|^2y^a\,dX \\
  &\leq \frac{2\varepsilon^2}{\lambda^{n+1+a}}+\frac{2N^2}{\lambda^{n+1+a}}
  \int_{B_\lambda^*}|X|^4y^a\,dX
  \leq \frac{2\varepsilon^2}{\lambda^{n+1+a}}+c_{n,a}\lambda^4.
 \end{aligned}
 \end{equation}

On the other hand, apply the trace inequality \eqref{eq:trace} to $U-\ell\in H^1(B_1^*,y^adX)$ to get,
for any $0<\lambda<1/8$,
\begin{equation}\label{trace ell}
 \lambda^{1+a}\int_{B_{\lambda}}|U(x,0)-\ell(x)|^2\,dx\leq C\left(\int_{B_{\lambda}^*}|U-\ell|^2y^a\,dX
 +\int_{B_{\lambda}^*}|\nabla(U-\ell)|^2y^a\,dX\right).
\end{equation}
Next we control the gradient in the right hand side of \eqref{trace ell}
by using the Caccioppoli inequality. Notice that $U-\ell$ is a solution
(in the sense of Definition \ref{def:weak}) to
$$\begin{cases}
\dive(y^aB(x)\nabla(U-\ell))=\dive(y^a(F+G)),&\hbox{in}~B_1^*,\\
-y^a(U-\ell)_y\big|_{y=0}=f,&\hbox{on}~B_1,
\end{cases}$$
where the vector field $G$ is given by
$$G=\left(\big(I-A(x)\big)\nabla_x\ell,0\right)\in \R^{n+1},\quad\hbox{and}~G(0)=0.$$
Then, by the Caccioppoli inequality in Lemma \ref{Lem:Caccioppoli},
\begin{align*}
 \int_{B_{\lambda}^*}&|\nabla(U-\ell)|^2y^a\,dX \leq 
 C\left(\int_{B_{2\lambda}^*}\big(|U-\ell|^2+|F+G|^2\big)y^a\,dX+
 \int_{B_{2\lambda}}|U(x,0)-\ell(x)||f(x)|\,dx\right)\\
 &\leq C\int_{B_{2\lambda}^*}|U-\ell|^2y^a\,dX+C\|F+G\|^2_{L^2(B_{2\lambda}^*,y^adX)}+
 C\big(\|U(\cdot,0)\|_{L^2(B_{2\lambda})}+\|\ell\|_{L^2(B_{2\lambda})}\big)\|{f}\|_{L^2(B_{2\lambda})}\\
 &\leq C\int_{B_{2\lambda}^*}|U-\ell|^2y^a\,dX
 +C\delta^2+C(1+D)\delta.
\end{align*}
Plugging this into \eqref{trace ell} and taking into account \eqref{primera} we see that, for any $0<\lambda<1/8$,
 \begin{align*}
 \frac{1}{\lambda^n}\int_{B_{\lambda}}|U(x,0)-\ell(x)|^2\,dx 
 &\leq \frac{C}{\lambda^{n+1+a}}\int_{B_{2\lambda}^*}|U-\ell|^2y^a\,dX
 +\frac{C}{\lambda^{n+1+a}}(\delta^2+\delta) \\
 &\leq \frac{C\varepsilon^2}{\lambda^{n+1+a}}+c_{n,a}\lambda^4+\frac{C\delta}{\lambda^{n+1+a}}.
 \end{align*}
 Hence, for any $0<\lambda<1/8$, from this and \eqref{primera},
  $$\frac{1}{\lambda^n}\int_{B_\lambda}|U(x,0)-\ell(x)|^2\,dx+\frac{1}{\lambda^{n+1+a}}
 \int_{B_\lambda^*}|U-\ell|^2y^a\,dX<\frac{C\varepsilon^2}{\lambda^{n+1+a}}
 +c_{n,a}\lambda^4+\frac{C\delta}{\lambda^{n+1+a}},$$
 where $C$ depends only on ellipticity, $n$ and $a$ and it is universal for any $W$.
 We first take $0<\lambda<1/8$ sufficiently small in such a way that the second term in the right hand side
 above is less than $\frac{1}{3}\lambda^{2(\alpha+2s)}$ (recall that we are in the situation
 where $1<\alpha+2s<2$). Then
 we let $\varepsilon>0$ small enough so that the first term is less than
 $\frac{1}{3}\lambda^{2(\alpha+2s)}$. For this choice of $\varepsilon$ we take $\delta>0$
 in the approximation lemma to be so small in such a way that the third term above is
 smaller than $\frac{1}{3}\lambda^{2(\alpha+2s)}$. Hence, 
 there exist a linear function $\ell(x)$, whose coefficients are bounded by a universal constant $D$, and $0<\delta<1$
 such that for any normalized solution and for some fixed $0<\lambda<1/8$,
   $$\frac{1}{\lambda^n}\int_{B_\lambda}|U(x,0)-\ell(x)|^2\,dx+\frac{1}{\lambda^{n+1+a}}
 \int_{B_\lambda^*}|U-\ell|^2y^a\,dX<\lambda^{2(\alpha+2s)}.$$
\end{proof}

\begin{lem}\label{lem:claim}
In the situation of Lemma \ref{a ver}, there exists a sequence of linear functions
$$\ell_k(x)=A_k+B_k\cdot x,$$
for $k\geq0$, such that
$$\frac{1}{\lambda^{kn}}\int_{B_{\lambda^k}}|U(x,0)-\ell_k(x)|^2\,dx+
\frac{1}{\lambda^{k(n+1+a)}}\int_{B_{\lambda^k}^*}|U-\ell_k|^2y^a\,dX<\lambda^{2k(\alpha+2s)},$$
and
$$|A_k-A_{k+1}|,\lambda^k|B_k-B_{k+1}|\leq D\lambda^{k(\alpha+2s)}.$$
\end{lem}

Before proceeding with the proof, let us show how this claim already implies \eqref{a probar ell L}. Let
$$\ell_\infty(x)=A_\infty+B_\infty\cdot x:=\left(\lim_{k\to\infty}A_k\right)+\left(\lim_{k\to\infty}B_k\right)\cdot x.$$
Notice that $A_\infty$ and $B_\infty$ are well defined because of the Cauchy property they verify.
Observe also that for any $k\geq0$, since $1<\alpha+2s<2$, we have
$$|\ell_\infty(x)-\ell_k(x)|\leq C_{\alpha,s}D\lambda^{k(\alpha+2s)},\quad |x|\leq\lambda^k.$$
For any $0<r<1/8$, take $k\geq0$ such that $\lambda^{k+1}<r\leq\lambda^{k}$.
Then, in a parallel way to \eqref{cuentita}, \eqref{a probar ell L} follows for small $r$.

\begin{proof}[Proof of Lemma \ref{lem:claim}]
 The proof is done by induction in $k\geq0$.
 When $k=0$, we take $\ell_0(x)=\ell_1(x)=0$ and the conclusion is true because $U$
 is a normalized solution. Assume that the claim is true for some $k\geq0$. Consider
$$\widetilde{U}(x,y)=\frac{U(\lambda^kx,\lambda^ky)-\ell_k(\lambda^kx)}{\lambda^{(\alpha+2s)k}},
\quad(x,y)\in B_1^*,$$
where $\lambda$ is as in Lemma \ref{a ver}.
Then, for $\widetilde{B}(x):=B(\lambda^k x)$, $\widetilde{\psi}(x)=\psi(\lambda^k x)$, $\widetilde{f}(x)=\lambda^{-k\alpha}f(\lambda^kx)$ and $\widetilde{F}(X)=\lambda^{-k(\alpha+2s-1)}F(\lambda^kX)$ we have
$$\int_{B_1^*}y^a\widetilde{B}(x)\nabla\widetilde{U}\nabla\widetilde{\psi}\,dX
=\int_{B_1}\widetilde{f}(x)\widetilde{\psi}(x,0)\,dx
+\int_{B_1^*}y^a\left(\widetilde{F}+\frac{I-\widetilde{B}(x)}{\lambda^{k(\alpha+2s-1)}}\nabla\ell_k
\right)\nabla\widetilde{\psi}\,dX.$$
Thus $\widetilde{U}$ is a weak solution to
$$\begin{cases}
\dive(y^a\widetilde{B}(x)\nabla\widetilde{U})=\dive(y^a(\widetilde{F}+\widetilde{G})),&\hbox{in}~B_1^*\\
-y^a\widetilde{U}_y\big|_{y=0}=\widetilde{f},&\hbox{on}~B_1,
\end{cases}$$
where
$$\widetilde{G}=\left(\frac{I-B(\lambda^kx)}{\lambda^{k(\alpha+2s-1)}}\nabla\ell_k,0\right)\in\R^{n+1},
\quad\hbox{and}~\widetilde{G}(0)=0.$$
Certainly $\widetilde{A}(0)=I$ and $\widetilde{f}(0)=0$.
By changing variables and using the hypotheses,
\begin{multline*}
\frac{1}{r^{n+1+a+2(\alpha+2s-1)}}\int_{B_r^*}y^a|\widetilde{F}+\widetilde{G}|^2\,dX\\
\leq \frac{1}{(\lambda^kr)^{n+1+a+2(\alpha+2s-1)}}\int_{B_{\lambda^kr}^*}y^a
\left(|F|^2+|I-B(x)|^2|B_k|^2\right)\,dX<(1+D^2c^2)\delta^2,
\end{multline*}
where we used that
$$|B_k|\leq\sum_{j=1}^k|B_j-B_{j-1}|\leq D\sum_{j=0}^\infty\lambda^{j(\alpha+2s-1)}=Dc.$$
Also,
\begin{align*}
\frac{1}{r^{n+2\alpha}}\int_{B_r}|\widetilde{f}|^2\,dx&=\frac{1}{(\lambda^kr)^{2\alpha+n}}
\int_{B_{\lambda^kr}}f^2\,dx\leq [f]_{L^{2,\alpha}(0)}^2<\delta^2;\\
\int_{B_1}\widetilde{U}(x,0)^2\,dx&=\frac{1}{\lambda^{kn}}
\int_{B_{\lambda^k}}\frac{|U(x,0)-\ell_k(x)|^2}{\lambda^{2k(\alpha+2s)}}\,dx\leq1;\\
\int_{B_1^*}y^a\widetilde{U}^2\,dX&=\frac{1}{\lambda^{k(n+1+a)}}
\int_{B_{\lambda^k}^*}y^a\frac{|U-\ell_k|^2}{\lambda^{2k(\alpha+2s)}}\,dX\leq1.
\end{align*}
By Lemma \ref{a ver} (that can be applied to $\widetilde{U}/(1+D^2c^2)$),
there is a linear function $\ell$ such that
 $$\frac{1}{\lambda^n}\int_{B_\lambda}|\widetilde{U}(x,0)-\ell(x)|^2\,dx+\frac{1}{\lambda^{n+1+a}}
 \int_{B_\lambda^*}|\widetilde{U}-\ell|^2y^a\,dX<\lambda^{2(\alpha+2s)},$$
So changing variables back in the definition of $\widetilde{U}$ we get
 $$\frac{1}{\lambda^{(k+1)n}}\int_{B_{\lambda^{k+1}}}|U(x,0)-\ell_{k+1}(x)|^2\,dx+\frac{1}{\lambda^{(k+1)(n+1+a)}}
 \int_{B_{\lambda^{k+1}}^*}|U-\ell_{k+1}|^2y^a\,dX<\lambda^{2(k+1)(\alpha+2s)},$$
where $\ell_{k+1}(x)=\ell_k(x)+\lambda^{k(\alpha+2s)}\ell(\lambda^{-k}x)$.
By construction,
$$|\ell_{k+1}(x)-\ell_k(x)|=\lambda^{k(\alpha+2s)}|\ell(\lambda^{-k}x)|\leq D\lambda^{k(\alpha+2s)}(1+\lambda^{-k}|x|).$$
When $x=0$ we get
$|A_{k+1}-A_k|\leq D\lambda^{k(\alpha+2s)}$.
On the other hand, again by construction,
$|B_{k+1}-B_k|\leq D\lambda^{-k}\lambda^{k(\alpha+2s)}$.
This finishes the proof.
\end{proof}

%%%%%%%%%%%%%%%%%%%%%%%%%%%%%%%%%%%%%%%%%%%%%%%%%%%%%%
\subsection{Proof of Theorem \ref{thm:interior Lmenos2alpha}}
%%%%%%%%%%%%%%%%%%%%%%%%%%%%%%%%%%%%%%%%%%%%%%%%%%%%%%

This proof is done by following exactly the same lines of the proof of Theorem \ref{thm:interior L2alpha},
but with easy changes in the exponents. Indeed, for part (1) we have to replace in the proof
of Theorem \ref{thm:interior L2alpha}(1) the exponent $\alpha$ that appears everywhere
there by $-2s+\alpha$. For part (2), along the proof of Theorem \ref{thm:interior L2alpha}(2)
we need to replace the exponent $\alpha$ by the new exponent $-2s+\alpha+1$.
Observe that we do not need the reduction to the case $f(0)=0$.

%%%%%%%%%%%%%%%%%%%%%%%%%%%%%%%%%%%%%%%%%%%%%%%%%%%%%%
\section{Case study: the fractional Dirichlet Laplacian in the half space}\label{section:5}
%%%%%%%%%%%%%%%%%%%%%%%%%%%%%%%%%%%%%%%%%%%%%%%%%%%%%%

In this section we study the global regularity and the growth near the boundary
for solutions to the fractional Dirichlet Laplacian of the half space.

%%%%%%%%%%%%%%%%%%%%%%%%%%%%%%%%%%%%%%%%%%%%%%%%%%%%%%
\subsection{Global regularity}
%%%%%%%%%%%%%%%%%%%%%%%%%%%%%%%%%%%%%%%%%%%%%%%%%%%%%%

Let us recall the Schauder estimates for the fractional Laplacian on $\R^n$.

\begin{prop}\label{Prop:Silvestre1}
 Let $0<s<1$ and $0<\alpha\leq1$. Assume that $f\in C^{0,\alpha}(\R^n)$ and that $u\in L^\infty(\R^n)$
 is a solution to 
 $$(-\Delta)^su=f,\quad\hbox{in}~\R^n.$$
 \begin{enumerate}[$(1)$]
  \item If $\alpha+2s<1$, then $u\in C^{0,\alpha+2s}(\R^n)$ and
  $$\|u\|_{C^{0,\alpha+2s}(\R^n)}\leq C\big(\|u\|_{L^\infty(\R^n)}+\|f\|_{C^{0,\alpha}(\R^n)}\big).$$
   \item If $1<\alpha+2s<2$, then $u\in C^{1,\alpha+2s-1}(\R^n)$ and
  $$\|u\|_{C^{1,\alpha+2s-1}(\R^n)}\leq C\big(\|u\|_{L^\infty(\R^n)}+\|f\|_{C^{0,\alpha}(\R^n)}\big).$$
   \item If $2<\alpha+2s<3$, then $u\in C^{2,\alpha+2s-2}(\R^n)$ and
  $$\|u\|_{C^{2,\alpha+2s-2}(\R^n)}\leq C\big(\|u\|_{L^\infty(\R^n)}+\|f\|_{C^{0,\alpha}(\R^n)}\big).$$
 \end{enumerate}
 The constants $C$ above depend only on $n$, $\alpha$ and $s$. In particular,
 if $f=0$ in a ball $B_r$, then $u$ is smooth in $B_{r/2}$.
\end{prop}

\begin{proof}
 For $(1)$ and $(2)$ see \cite[Proposition~2.8]{Silvestre CPAM}. 
 The statement in $(3)$ is proved analogously by taking into account the range of the exponents.
 The details are omitted.
\end{proof}

Recall that the Zygmund space $\Lambda_\ast(\R^n)$ consists of all bounded functions $u$ on $\R^n$ such that
$$[u]_{\Lambda_\ast(\R^n)}:=\sup_{x,h\in\R^n}\frac{|u(x+h)-2u(x)+u(x-h)|}{|h|}<\infty,$$
under the norm $\|u\|_{\Lambda_\ast(\R^n)}:=\|u\|_{L^\infty(\R^n)}+[u]_{\Lambda_\ast(\R^n)}$, see \cite{Zygmund}
(also \cite[Chapter~V]{Stein} or \cite{Roncal-Stinga}).

\begin{prop}\label{Prop:Silvestre2}
 Let $0<s<1$. Assume that $f\in L^\infty(\R^n)$ and that $u\in L^\infty(\R^n)$
 is a solution to 
 $$(-\Delta)^su=f,\quad\hbox{in}~\R^n.$$
 Namely, assume that $u\in L^\infty(\R^n)$ is given by
 $$u(x)=(-\Delta)^{-s}f(x)=\frac{1}{\Gamma(s)}\int_0^\infty e^{t\Delta}f(x)\,\frac{dt}{t^{1-s}},$$
 where the integral is well defined for almost all $x\in\R^n$ and for some $f\in L^\infty(\R^n)$.
 \begin{enumerate}[$(1)$]
  \item If $0<2s<1$ then $u\in C^{0,2s}(\R^n)$ and
  $$\|u\|_{C^{0,2s}(\R^n)}\leq C\big(\|u\|_{L^\infty(\R^n)}+\|f\|_{L^\infty(\R^n)}\big).$$
  \item If $2s=1$ then $u$ is in the Zygmund space $\Lambda_*(\R^n)$ and
  $$\|u\|_{\Lambda_*(\R^n)}\leq C\big(\|u\|_{L^\infty(\R^n)}+\|f\|_{L^\infty(\R^n)}\big).$$
   \item If $1<2s<2$ then $u\in C^{1,2s-1}(\R^n)$ and
  $$\|u\|_{C^{1,2s-1}(\R^n)}\leq C\big(\|u\|_{L^\infty(\R^n)}+\|f\|_{L^\infty(\R^n)}\big).$$
 \end{enumerate}
 The constants $C$ above depend only on $n$ and $s$.
\end{prop}

\begin{proof}
Parts (1) and (3) of this result, that is, when $2s\neq1$, are already contained in \cite[Theorem~6.4]{Allen}\footnote{We are grateful to Mark Allen for pointing out this to us.}. Here we
present a proof that works for every $0<s<1$ and includes the Zygmund space.

For $\alpha>0$, let $\Lambda_\alpha$ be the space of bounded functions $u$ on $\R^n$ for which
$$[u]_{\Lambda_\alpha}:=\sup_{x\in\R^n,t>0}|t^{1-\alpha/2}\partial_te^{t\Delta}u(x)|<\infty.$$
It is known that
$$\Lambda_\alpha=\begin{cases}
C^{0,\alpha}(\R^n),&\hbox{if}~0<\alpha<1,\\
\Lambda_\ast(\R^n),&\hbox{if}~\alpha=1,\\
C^{1,\alpha-1}(\R^n),&\hbox{if}~1<\alpha<2.
\end{cases}$$
Moreover, the norm on all these spaces is equivalent to $\|u\|_{L^\infty(\R^n)}+[u]_{\Lambda_\alpha}$.
See \cite{Roncal-Stinga} where this result is proved for the torus. In \cite{Stein} a similar characterization is proved by
using the Poisson semigroup instead of the heat semigroup.
The proof in \cite{Roncal-Stinga} can be easily adapted to the case of $\R^n$.
It is enough to show that 
$$[u]_{\Lambda_{2s}}=\sup_{x\in\R^n,t>0}|t^{1-s}\partial_te^{t\Delta}(-\Delta)^{-s}f(x)|\leq C\|f\|_{L^\infty(\R^n)},$$
for some constant $C$ depending only on $n$ and $s$.
Consider the heat kernel $W_t(x)=(4\pi t)^{-n/2}e^{-|x|^2/(4t)}$, $x\in\R^n$, $t>0$. Notice that the following simple estimate
$$\int_{\R^n}|\partial_tW_t(x)|\,dx\leq\frac{c}{t},\quad t>0,$$
implies that
$$|\partial_te^{t\Delta}f(x)|\leq \frac{c}{t}\|f\|_{L^\infty(\R^n)},\quad\hbox{for all}~x\in\R^n,~t>0.$$
Thus, with this and the semigroup property $e^{t\Delta}e^{r\Delta}f=e^{(t+r)\Delta}f$ we obtain
\begin{align*}
	|t\partial_te^{t\Delta}(-\Delta)^{-s}f(x)| &\leq Ct\int_0^\infty\bigg|\partial_we^{w\Delta}f(x)\big|_{w=t+r}\bigg|\,\frac{dr}{r^{1-s}} \\
	&\leq Ct\|f\|_{L^\infty(\R^n)}\int_0^\infty \frac{1}{t+r}\,\frac{dr}{r^{1-s}} \\
	&= Ct^s\|f\|_{L^\infty(\R^n)}\int_0^\infty \frac{1}{1+\rho}\,\frac{d\rho}{\rho^{1-s}}=C\|f\|_{L^\infty(\R^n)}t^s,
\end{align*}
where $C$ is a constant that depends only on $n$ and $s$.
\end{proof}

In the half space $\R^n_+$ we consider the
Laplacian with homogeneous Dirichlet boundary condition on $\partial\R^n_+=\set{x_n=0}$.
We denote this operator by $-\Delta_D^+$. Then $-\Delta_D^+$ is a nonnegative and 
selfadjoint operator on $H^1_0(\R^n_+)$ for which the spectral theorem applies. For
a function $u$ defined on 
$\overline{\R^n_+}$ with $u(x',0)=0$ and for $0<s<1$ we have
\begin{equation}\label{Dir frac Lap heat}
 (-\Delta_D^+)^su(x)=\frac{1}{\Gamma(-s)}\int_0^\infty\big(e^{t\Delta_D^+}u(x)-u(x)\big)\frac{dt}{t^{1+s}},
 \quad x\in\R^n_+.
\end{equation}
Here $v(x,t)\equiv e^{t\Delta_D^+}u(x)$ is the heat semigroup generated by $-\Delta_D^+$
on the half space, namely, $v$ is the solution to
\begin{equation*}
\begin{cases}
  v_t=\Delta v,&\hbox{for}~x\in\R^{n}_+,t>0,\\
  v(x,0)=u(x),&\hbox{on}~\R^n_+,\\
  v(x',0,t)=0,&\hbox{for}~t\geq0.
\end{cases}
\end{equation*}
Let $x^*=(x',-x_n)$, for $x\in\R^n$.
Denote by $u_{o}$ the odd extension of $u$ to $\R^n$ with respect to $x_n$:
\begin{equation*}
 u_{o}(x)=
\begin{cases}
 u(x),&\hbox{if}~x_n\geq0,\\
 -u(x^*)=-u(x',-x_n),&\hbox{if}~x_n<0.
\end{cases}
\end{equation*}
By using the reflection method we see that
$$e^{t\Delta_D^+}u(x)=e^{t\Delta}u_{o}(x),\quad t>0,~x\in\R^n_+,$$
so that, from \eqref{Dir frac Lap heat} we observe that
\begin{equation}\label{relation}
(-\Delta_D^+)^{s}u(x)=(-\Delta)^{s}u_{o}(x),\quad x\in\R^n_+,
\end{equation}
where $(-\Delta)^s$ is the fractional Laplacian on $\R^n$. Moreover, since
$$e^{t\Delta_D^+}u(x)=\frac{1}{(4\pi t)^{n/2}}\int_{\R^{n}_+}\big(e^{-|x-z|^2/(4t)}-e^{-|x-z^*|^2/(4t)}\big)u(z)\,dz,
\quad x\in\R^n_+,$$
from \eqref{Dir frac Lap heat} we obtain the following pointwise formula:
$$(-\Delta_D^+)^{s}u(x)=c_{n,s}\int_{\R^n_+}\big(u(x)-u(z)\big)
  \left(\frac{1}{|x-z|^{n+2s}}-\frac{1}{|x-z^\ast|^{n+2s}}\right)\,dz,\quad x\in\R^n_+.$$
Also, from the fact that 
$$(-\Delta_D^+)^{-s}f(x)=\frac{1}{\Gamma(s)}\int_0^\infty e^{t\Delta_D}f(x)\,\frac{dt}{t^{1-s}},$$
we get, when $n\neq 2s$,
\begin{equation}\label{inverso}
(-\Delta_D^+)^{-s}f(x)=d_{n,s}\int_{\R^n_+}f(z)
  \left(\frac{1}{|x-z|^{n-2s}}-\frac{1}{|x-z^\ast|^{n-2s}}\right)\,dz,\quad x\in\R^n_+.
  \end{equation}
  The constants $c_{n,s}$ and $d_{n,s}$ above can be computed explicitly.

\begin{thm}[Global regularity in half space -- Dirichlet case]\label{Thm:Laplacian half space}
Let $u$ be a bounded solution to
$$\begin{cases}
   (-\Delta_D^+)^su=f,&\hbox{in}~\R^n_+,\\
   u=0,&\hbox{on}~\partial\R^n_+=\{x_n=0\},
  \end{cases}$$
where $f\in C^{0,\alpha}(\overline{\R^n_+})$, $0<\alpha\leq1$.
\begin{itemize}
\item Suppose that $f(x',0)=0$, for all $x'\in\R^{n-1}$. Then
\begin{itemize}
 \item[(1)] If $\alpha+2s<1$ then $u\in C^{0,\alpha+2s}(\overline{\R^n_+})$ and 
 $$\|u\|_{C^{0,\alpha+2s}(\overline{\R^n_+})}\leq C\big(\|u\|_{L^\infty(\R^n_+)}+\|f\|_{C^{0,\alpha}(\overline{\R^n_+})}\big).$$
 \item[(2)] If $1<\alpha+2s<2$ then $u\in C^{1,\alpha+2s-1}(\overline{\R^n_+})$ and
 $$\|u\|_{C^{1,\alpha+2s-1}(\overline{\R^n_+})}\leq C\big(\|u\|_{L^\infty(\R^n_+)}+\|f\|_{C^{0,\alpha}(\overline{\R^n_+})}\big).$$
 \item[(3)] If $2<\alpha+2s<3$ then $u\in C^{2,\alpha+2s-2}(\overline{\R^n_+})$ and
 $$\|u\|_{C^{2,\alpha+2s-2}(\overline{\R^n_+})}\leq C\big(\|u\|_{L^\infty(\R^n_+)}+\|f\|_{C^{0,\alpha}(\overline{\R^n_+})}\big).$$
\end{itemize}
\item If $f(x',0)\neq0$ at some $x'\in\R^{n-1}$ then
\begin{itemize}
\item[(i)] If $0<2s<1$ then $u\in C^{0,2s}(\overline{\R^n_+})$ and
 $$\|u\|_{C^{0,2s}(\overline{\R^n_+})}\leq C\big(\|u\|_{L^\infty(\R^n_+)}+\|f\|_{L^\infty(\R^n_+)}\big).$$
 \item[(ii)] If $2s=1$ then $u$ is in the Zygmund space $\Lambda_\ast(\overline{\R^n_+})$ and
 $$\|u\|_{\Lambda_\ast(\overline{\R^n_+})}\leq C\big(\|u\|_{L^\infty(\R^n_+)}+\|f\|_{L^\infty(\R^n_+)}\big).$$
 \item[(iii)] If $1<2s<2$ then $u\in C^{1,2s-1}(\overline{\R^n_+})$ and
 $$\|u\|_{C^{1,2s-1}(\overline{\R^n_+})}\leq C\big(\|u\|_{L^\infty(\R^n_+)}+\|f\|_{L^\infty(\R^n_+)}\big).$$
\end{itemize}
\end{itemize}
All the constants $C$ above depend only on $n$, $\alpha$ and $s$.
\end{thm}

\begin{proof}
From \eqref{relation} we see that $(-\Delta)^su_o(x)=(-\Delta_D^+)^su(x)=f(x)$ when $x\in\R^n_+$.
On the other hand, for $x=(x',x_n)$ with $x_n<0$ we have
\begin{align*}
	(-\Delta)^su_o(x) &= \frac{1}{\Gamma(-s)}\int_0^\infty\big(e^{t\Delta}u_o(x)-u_{o}(x)\big)\frac{dt}{t^{1+s}} \\
	 &= \frac{1}{\Gamma(-s)}\int_0^\infty\big(u(x^*)-e^{t\Delta_D^+}u(x^*)\big)\frac{dt}{t^{1+s}} \\
	 &= -(-\Delta_D^+)^su(x^*) = -f(x^*).
\end{align*}
Hence,
\begin{equation}\label{relation2}
(-\Delta)^{s}u_{o}(x)=f_{o}(x),\quad\hbox{for all}~x\in\R^n.
\end{equation}
Now we apply the results by Silvestre to $u_{o}$ (which coincides with $u$ when $x_n\geq0$).
For (1)--(3), we notice that the condition $f(x',0)=0$ ensures that the odd extension $f_{o}$
is globaly in $C^{0,\alpha}(\R^n)$.
Then we can recall Proposition \ref{Prop:Silvestre1}. As for (i)--(iii),
we can only assure that $f_{o}$ is just bounded (it has a jump discontinuity
at $x_n=0$) and the conclusion follows from Propostion \ref{Prop:Silvestre2}.
\end{proof}

%%%%%%%%%%%%%%%%%%%%%%%%%%%%%%%%%%%%%%%%%%%%%%%%%%%%%%
\subsection{Particular one dimensional solutions}
%%%%%%%%%%%%%%%%%%%%%%%%%%%%%%%%%%%%%%%%%%%%%%%%%%%%%%

In this subsection we study the growth near the boundary
of solutions to the one dimensional fractional problem
$$\begin{cases}
(-D_{xx}^2)^su(x)=f(x),&\hbox{for}~x>0,\\
u(0)=0,
\end{cases}$$
where $-D_{xx}^2$ denotes the Dirichlet Laplacian on the half line $[0,\infty)$ and
$$f=\begin{cases}
1,&\hbox{when}~0<s<1/2,\\
\chi_{[0,1]}(x),&\hbox{when}~1/2\leq s<1.
\end{cases}$$

%%%%%%%%%%%%%%%%%%%%%%%%%%%%%%%%%%%%%%%%%%%%%%%%%%%%%%
\subsubsection{Case $0<s<1/2$ and $f\equiv1$}
%%%%%%%%%%%%%%%%%%%%%%%%%%%%%%%%%%%%%%%%%%%%%%%%%%%%%%

From \eqref{inverso},
\begin{align*}
	u(x) &= c_s\int_0^\infty\bigg(\frac{1}{|x-z|^{1-2s}}-\frac{1}{|x+z|^{1-2s}}\bigg)\,dz \\
	&= \frac{c_s}{x^{1-2s}}\int_0^\infty\bigg(\frac{1}{|1-z/x|^{1-2s}}-\frac{1}{|1+z/x|^{1-2s}}\bigg)\,dz \\
	&= c_sx^{2s}\int_0^\infty\bigg(\frac{1}{|1-\omega|^{1-2s}}-\frac{1}{|1+\omega|^{1-2s}}\bigg)\,d\omega.
\end{align*}
The last integral above is finite. Indeed, since $s>0$, the integral converges at the origin.
On the other hand, let $\omega>2$. Consider the function $\varphi(t)=(\omega-t)^{2s-1}$, for $-1\leq t\leq1$. Then
$$\varphi(1)-\varphi(-1)=|1-\omega|^{2s-1}-|1+\omega|^{2s-1}=2\varphi'(\xi)\leq C_s\omega^{2s-2},$$
which implies that the integral converges at infinity for $s<1/2$. We conclude that 
$$u(x)=c_sx^{2s},\quad x\in\R^+,~0<s<1/2,$$
for some positive constant $c_s$ that can be computed explicitly.

%%%%%%%%%%%%%%%%%%%%%%%%%%%%%%%%%%%%%%%%%%%%%%%%%%%%%%
\subsubsection{Case $s=1/2$ and $f=\chi_{[0,1]}$}
%%%%%%%%%%%%%%%%%%%%%%%%%%%%%%%%%%%%%%%%%%%%%%%%%%%%%%

Let $0<x<1/2$. We can write
\begin{align*}
	u(x) &= c\int_0^1\big(\ln|x+z|-\ln|x-z|\big)\,dz 
	= c\int_0^1\big(\ln|x(1+z/x)|-\ln|x(1-z/x)|\big)\,dz \\
	&= c\int_0^1\big(\ln|1+z/x|-\ln|1-z/x|\big)\,dz 
	= cx\int_0^{1/x}\big(\ln|1+\omega|-\ln|1-\omega|\big)\,d\omega \\
	&= cx\left(C+\int_2^{1/x}\big(\ln(\omega+1)-\ln(\omega-1)\big)\,d\omega\right) \\
	&= cx\left[C+\left(\tfrac{1}{x}+1\right)\ln\left(\tfrac{1}{x}+1\right)
	-\left(\tfrac{1}{x}-1\right)\ln\left(\tfrac{1}{x}-1\right)\right] \\
	&=c\left[Cx+(1+x)\ln(1+x)-(1-x)\ln(1-x)-2x\ln x\right].
\end{align*}
It is clear that
$$u(x)=-2cx\ln x+w(x),\quad 0<x<1/2.$$
where $w$ is smooth up to $x=0$.
Recall that, $\ln(1+x)\sim x$ and $-\ln(1-x)\sim x$, when $0<x<1/2$.
Therefore,
$$u(x)\sim x\ln x,\quad\hbox{as}~x\to0^+.$$

%%%%%%%%%%%%%%%%%%%%%%%%%%%%%%%%%%%%%%%%%%%%%%%%%%%%%%
\subsubsection{Case $1/2<s<1$ and $f=\chi_{[0,1]}$}
%%%%%%%%%%%%%%%%%%%%%%%%%%%%%%%%%%%%%%%%%%%%%%%%%%%%%%

From \eqref{inverso},
\begin{align*}
	u(x) &= c_s\int_0^1\big(|x-y|^{2s-1}-(x+y)^{2s-1}\big)\,dy \\
	&= c_s\left(\int_0^x(x-y)^{2s-1}\,dy+\int_x^1(y-x)^{2s-1}\,dy-\tfrac{1}{2s}\big((x+1)^{2s}
	-x^{2s}\big)\right) \\
	&= \frac{c_s}{2s}\left(2x^{2s}+(1-x)^{2s}-(1+x)^{2s}\right).
\end{align*}
It is clear that, for some constant $c>0$,
$$u(x)=cx^{2s}+w(x),\quad 0<x<1/2,$$
where $w$ is smooth up to $x=0$.
By taking into account the series expansions of $(1\pm x)^{2s}$ it is easy to check that
$$u(x)\sim x,\quad\hbox{as}~x\to0^+.$$

%%%%%%%%%%%%%%%%%%%%%%%%%%%%%%%%%%%%%%%%%%%%%%%%%%%%%%
\subsection{Behavior near the boundary for half space solutions}
%%%%%%%%%%%%%%%%%%%%%%%%%%%%%%%%%%%%%%%%%%%%%%%%%%%%%%

Our next step is to consider the problem for the fractional Dirichlet Laplacian in the half
space $\R^n_+$, $n\geq2$,
\begin{equation}\label{multi}
\begin{cases}
(-\Delta_D^+)^sw=g,&\hbox{in}~\R^n_+,\\
w(x',0)=0,&\hbox{on}~\partial\R^n_+,
\end{cases}
\end{equation}
in the cases where
\begin{equation}\label{f}
g(x)=g(x',x_n)=\begin{cases}
1,&\hbox{when}~0<s<1/2,\\
\chi_{[0,1]}(x_n),&\hbox{when}~1/2\leq s<1.
\end{cases}
\end{equation}
To that end we apply the following result.

\begin{lem}
Let $g:\R^n\to\R$ be a function depending only on the $x_n$--variable, that is,
$g(x)=\phi(x_n)$ for some function $\phi:\R\to\R$, for all $x\in\R^n$. Then the solution to
$$(-\Delta)^sw=g,\quad\hbox{in}~\R^n,$$
is a function that depends only on $x_n$. More precisely, $w(x)=\psi(x_n)$
for all $x\in\R^n$, where $\psi:\R\to\R$ is the solution to the one dimensional problem
$$(-\partial_{xx}^2)^s\psi=\phi,\quad\hbox{in}~\R^n.$$
Here $-\partial_{xx}^2$ is the Laplacian on the real line $\R$.
\end{lem}

\begin{proof}
Notice first that
\begin{align*}
	e^{t\Delta}g(x) &= \int_{-\infty}^\infty\frac{e^{-|x_n-z_n|^2/(4t)}}{(4\pi t)^{n/2}}\phi(z_n)
	\left(\int_{\R^{n-1}}e^{-|x'-z'|^2/(4t)}\,dz'\right)\,dz_n \\
	&= \int_{-\infty}^\infty\frac{e^{-|x_n-z_n|^2/(4t)}}{(4\pi t)^{1/2}}\phi(z_n)\,dz_n
	= e^{-t\partial_{xx}^2}\phi(x_n),
\end{align*}
where $\{e^{-t\partial_{xx}^2}\}_{t>0}$ denotes the heat semigroup on the real line. Hence,
\begin{align*}
	w(x) &= (-\Delta)^{-s}g(x)=\frac{1}{\Gamma(s)}\int_0^\infty e^{t\Delta}g(x)\,\frac{dt}{t^{1-s}} \\
	 &= \frac{1}{\Gamma(s)}\int_0^\infty e^{-t\partial_{xx}^2}\phi(x_n)\,\frac{dt}{t^{1-s}}
	 =(-\partial_{xx}^2)^{-s}\phi(x_n)=\psi(x_n).
\end{align*}
\end{proof}

In view of the previous Lemma, the one dimensional results
and the relations \eqref{relation} and \eqref{relation2},
we get that the solution $w$ to \eqref{multi} with $g$ as in 
\eqref{f} satisfy the following properties:
\begin{equation}\label{solutions}
w(x)=\begin{cases}
cx_n^{2s},&\hbox{for all}~x\in\R^n_+,~\hbox{when}~0<s<1/2,\\
-cx_n\ln x_n+\eta_{1/2}(x_n),&\hbox{for all}~x\in\R^n_+,~x_n<1/2,~\hbox{when}~s=1/2,\\
cx_n^{2s}+\eta(x_n),&\hbox{for all}~x\in\R^n_+,~x_n<1/2,~\hbox{when}~1/2<s<1.
\end{cases}
\end{equation}
In the last two cases, $\eta_{1/2}$ and $\eta$ are smooth up to $x_n=0$. Also,
\begin{equation}\label{crecimiento al borde}
w(x)\sim x_n^{\min\{2s,1\}},\quad\hbox{as}~x_n\to0^+,~\hbox{uniformly in}~x'\in\R^{n-1},
~\hbox{if}~s\neq1/2,
\end{equation}
and
$$w(x)\sim x_n|\ln x_n|,\quad\hbox{as}~x_n\to0^+,~\hbox{uniformly in}~x'\in\R^{n-1},
~\hbox{if}~s=1/2.$$
Finally, it is clear that the solution $W$ to
\begin{equation}\label{extension w}
\begin{cases}
  \dive(y^{a}\nabla W)=0,& \hbox{in}~\R^n_+\times(0,\infty),\\
  -y^{a}W_y\big|_{y=0}=\theta g,& \hbox{on}~\R^n_+,\\
  W=0,&\hbox{on}~\partial\R^n_+\times[0,\infty),
\end{cases}
\end{equation}
with $g$ as in \eqref{f} and $\theta\in\R$ satisfies
$${W}(x,0)=\theta w(x),\quad x\in\R^n_+.$$

%%%%%%%%%%%%%%%%%%%%%%%%%%%%%%%%%%%%%%%%%%%%%%%%%%%%%%
\section{Boundary regularity -- Dirichlet case}\label{section:6}
%%%%%%%%%%%%%%%%%%%%%%%%%%%%%%%%%%%%%%%%%%%%%%%%%%%%%%

Theorems \ref{thm:boundary Calpha} and \ref{thm:boundary Lp}
for solutions to \eqref{uma} are consequences of 
the more general results that we state and prove here.

Throughout this section we assume that 
$\Omega\subset\R^n_+$ is a bounded domain whose boundary $\partial\Omega$ contains a flat portion
on $\{x_n=0\}$ in such a way that $B_1^+\subset\Omega$.

We say that a function $f:B_1\cap\{x_n\geq0\}\to\R$ is in $L_{\partial\Omega}^{2,\alpha}(0)$, for $0\leq\alpha<1$, whenever
$$[f]_{L_{\partial\Omega}^{2,\alpha}(0)}^2:=\sup_{0<r\leq1}\frac{1}{r^{n+2\alpha}}\int_{B_r^+}|f(x)-f(0)|^2\,dx<\infty,$$
where $f(0)$ is defined as  $\displaystyle f(0):=\lim_{r\to0}\frac{1}{|B_r^+|}\int_{B_r^+}f(x)\,dx$.
It is clear that if $f$ is H\"older continuous of order $\alpha$ at $0$ then $f\in L_{\partial\Omega}^{2,\alpha}(0)$.
If the condition above holds uniformly in balls centered at points of $\partial\Omega$
close to the origin, then $f$ is $\alpha$--H\"older continuous at the
boundary near the origin, see \cite{Campanato}.

\begin{thm}\label{thm:boundary L2alpha}
Consider the half space solutions $w$ given in \eqref{solutions}.
Let $u$ be a solution to \eqref{uma}. Assume that $f\in L_{\partial\Omega}^{2,\alpha}(0)$,
for some $0<\alpha<1$.
 \begin{enumerate}[$(1)$]
  \item Suppose that $0<\alpha+2s<1$.
  There exist $0<\delta<1$, depending only on $n$, ellipticity, $\alpha$ and $s$,
   and a constant $C_0>0$ such that if
   $$\sup_{0<r\leq1}\frac{1}{r^n}\int_{B_r^+}|A(x)-A(0)|^2\,dx<\delta^2,$$
   then
  $$\frac{1}{r^n}\int_{B_r^+}|u(x)-f(0)w(x)|^2\,dx\leq C_1r^{2(\alpha+2s)},\quad\hbox{for all}~r>0~
  \hbox{sufficiently small},$$
  where $C_1\leq C_0\big(1+\|u\|_{L^2(\Omega)}+[u]_{H^s(\Omega)}+[f]_{L_{\partial\Omega}
  ^{2,\alpha}(0)}+|f(0)|\big)$.
   \item Suppose that $s\geq1/2$ and $1<\alpha+2s<2$.
   There exist $0<\delta<1$, depending only on $n$, ellipticity, $\alpha$ and $s$,
   and a constant $C_0>0$ such that if
   $$\sup_{0<r\leq1}\frac{1}{r^{n+2(\alpha+2s-1)}}\int_{B_r^+}|A(x)-A(0)|^2\,dx<\delta^2,$$
   then there exists a linear function $\ell(x)=\mathcal{B}\cdot x$ such that
     $$\frac{1}{r^n}\int_{B_r^+}|u(x)-f(0)w(x)-\ell(x)|^2\,dx\leq C_1r^{2(\alpha+2s)},\quad\hbox{for all}~r>0~
     \hbox{sufficiently small},$$
  where $C_1+|\mathcal{B}|
  \leq C_0\big(1+\|u\|_{L^2(\Omega)}+[u]_{H^s(\Omega)}+[f]_{L_{\partial\Omega}^{2,\alpha}(0)}+|f(0)|\big)$.
 \end{enumerate}
 The constants $C_0$ above depend only on $[A]_{L^{2,0}_{\partial\Omega}(0)}$
 (resp. $[A]_{L^{2,\alpha+2s-1}_{\partial\Omega}(0)}$), ellipticity, $n$, $\alpha$ and $s$.
\end{thm}

Observe the extra term $1$ in the estimates for $C_1$ and $C_1+|\mathcal{B}|$ that comes from
the $H^s$ norm of $w$.

Theorem \ref{thm:boundary Calpha} is a consequence of Theorem \ref{thm:boundary L2alpha}.
Indeed, first notice that 
the conclusion of Theorem \ref{thm:boundary L2alpha} can be translated to any point 
$x_0$ at $\partial\Omega$. We can first flatten
the boundary of $\Omega$ at $x_0$
and then rescale (and rotate if necessary)
the resulting domain so that $B_1^+(x_0)\subset\Omega$ with $B_1(x_0)\cap\{x_n=0\}\subset\partial\Omega$. Then
$v:=u-f(x_0)w$ has the desired regularity around $x_0$ (remember that $u(x_0)-f(x_0)w(x_0)=0$).
By taking into account the growth of $w$ near the boundary \eqref{crecimiento al borde} and
going back to the initial variables the conclusion follows.

In a similar way as we did for $L^{2,\alpha}_{\partial\Omega}(0)$, we can define
$L^{2,-2s+\alpha}_{\partial\Omega}(0)$ and $L^{2,-2s+\alpha+1}_{\partial\Omega}(0)$.
It is clear that if $f\in L^p(B_1^+)$ then parallel remarks as those preceding Theorem 
\ref{thm:interior Lmenos2alpha} hold for $L^{2,-2s+\alpha}_{\partial\Omega}(0)$
and $L^{2,-2s+\alpha+1}_{\partial\Omega}(0)$, with the same exponents $p$ and $\alpha$.

\begin{thm}\label{thm:boundary Lmenos2alpha}
Let $u$ be a solution to \eqref{uma}.
\begin{enumerate}[$(1)$]
  \item Suppose that $f\in L_{\partial\Omega}^{2,-2s+\alpha}(0)$.
  There exist $0<\delta<1$, depending only on $n$, ellipticity, $\alpha$ and $s$,
   and a constant $C_0>0$ such that if
   $$\sup_{0<r\leq1}\frac{1}{r^n}\int_{B_r^+}|A(x)-A(0)|^2\,dx<\delta^2,$$
   then
  $$\frac{1}{r^n}\int_{B_r^+}|u(x)|^2\,dx\leq C_1r^{2\alpha},\quad\hbox{for all}~
  r>0~\hbox{sufficiently small},$$
  where $C_1\leq C_0\big(\|u\|_{L^2(\Omega)}+[u]_{H^s(\Omega)}+[f]_{L^{2,-2s+\alpha}_{\partial\Omega}(0)}\big)$.
   \item Suppose that $f\in L_{\partial\Omega}^{2,-2s+\alpha+1}(0)$. 
   There exist $0<\delta<1$, depending only on $n$, ellipticity, $\alpha$ and $s$,
    and a constant $C_0>0$ such that if
   $$\sup_{0<r\leq1}\frac{1}{r^{n+2\alpha}}\int_{B_r^+}|A(x)-A(0)|^2\,dx<\delta^2,$$
   then there exists a linear function $\ell(x)=\mathcal{B}\cdot x$ such that
     $$\frac{1}{r^n}\int_{B_r^+}|u(x)-\ell(x)|^2\,dx\leq C_1r^{2(1+\alpha)},\quad\hbox{for all}
     ~r>0~\hbox{sufficiently small},$$
  where $C_1+|\mathcal{B}|
  \leq C_0\big(\|u\|_{L^2(\Omega)}+[u]_{H^s(\Omega)}+[f]_{L_{\partial\Omega}^{2,-2s+\alpha+1}(0)}\big)$.
 \end{enumerate}
The constants $C_0$ above depend only on $[A]_{L^{2,0}_{\partial\Omega}(0)}$
 (resp. $[A]_{L^{2,\alpha}_{\partial\Omega}(0)}$), ellipticity, $n$, $\alpha$ and $s$.
\end{thm}

Notice that in Theorem \ref{thm:boundary Lmenos2alpha} we do not need to subtract $w$ from $u$
to obtain the regularity up to the origin. Observe that Theorem \ref{thm:boundary Lp}
in the Dirichlet case follows from this last result
after flattening the boundary.

The rest of this section is devoted to the proof of Theorems \ref{thm:boundary L2alpha} and \ref{thm:boundary Lmenos2alpha}.

%%%%%%%%%%%%%%%%%%%%%%%%%%%%%%%%%%%%%%%%%%%%%%%%%%%%%%
\subsection{Proof of Theorem \ref{thm:boundary L2alpha}(1)}
%%%%%%%%%%%%%%%%%%%%%%%%%%%%%%%%%%%%%%%%%%%%%%%%%%%%%%

It is enough to prove the result for $u(x)=U(x,0)$, where
$U\in H^1((B_1^+)^*,y^adX)$ is a solution to
\begin{equation}\label{dale que va}
\begin{cases}
  \dive(y^{a}B(x)\nabla U)=\dive(y^aF),& \hbox{in}~(B_1^+)^*,\\
  -y^{a}U_y\big|_{y=0}=f,& \hbox{on}~B_1^+,\\
  U=0,&\hbox{on}~B_1\cap\{x_n=0\}.
\end{cases}
\end{equation}
Here $F$ is a $(B_1^+)^*$--valued vector field such that $F_{n+1}=0$ and satisfies the Morrey condition
 $$\sup_{0<r\leq1}\frac{1}{r^{n+1+a+2(\alpha+2s-1)}}\int_{(B_r^+)^*}y^a|F|^2\,dX<\infty.$$
After a change of variables we can assume that $B(0)=I$.

We compare $U(x,0)$ with $W(x,0)$, where $W$ is the solution to \eqref{extension w} with $\theta=f(0)$.
In particular, $W\in H^1((B_1^+)^*,y^adX)$ is a solution to
$$\begin{cases}
  \dive(y^{a}\nabla W)=0,& \hbox{in}~(B_1^+)^*,\\
  -y^{a}W_y\big|_{y=0}=f(0),& \hbox{on}~B_1^+,\\
  W=0,&\hbox{on}~B_1\cap\{x_n=0\},
\end{cases}$$
and $W(x,0)=f(0)w(x)$, for $x\in B_1^+$, with $w$ as in \eqref{solutions}.

Let $V=U-W$. Then
\begin{equation}\label{V problem}
\begin{cases}
  \dive(y^{a}B(x)\nabla V)=\dive(y^aH),& \hbox{in}~(B_1^+)^*,\\
  -y^{a}V_y\big|_{y=0}=h,& \hbox{on}~B_1^+,\\
  V=0,&\hbox{on}~B_1\cap\{x_n=0\},
\end{cases}
\end{equation}
where $h=f-f(0)$, so that $h(0)=0$, and $H=F+(I-B(x))\nabla W$, with $H_{n+1}=0$.
Given $\delta>0$ we can always assume that the conditions $(1)$--$(4)$
in Subsection \ref{subsection:proof interior(1)} hold with the appropriate changes:
$B_r^+$, $h$, $H$, $V$, $B_1^+$ and $(B_1^+)^*$ in place of $B_r$, $f$, $F$, $U$, $B_1$ and $B_1^*$, respectively.
Under those hypotheses and the proper normalization for the $L^2$ norms of $V$
and its trace, we call $V$ a normalized solution.

Now we prove that given $0<\alpha+2s<1$ there exists $0<\delta<1$,
depending only on $n$, ellipticity, $\alpha$ and $s$, such that for any normalized solution $V$ of \eqref{V problem}
(recall that $V(0,0)=0$)
$$\frac{1}{r^n}\int_{B_r^+}|V(x,0)|^2\,dx\leq Cr^{2(\alpha+2s)},\quad\hbox{for all}~r>0~\hbox{sufficiently
small},$$
and $C\leq C_0$, for some constant $C_0$ depending only on $n$, ellipticity, $\alpha$ and $s$.
The strategy of the proof is the same as the proof of Theorem \ref{thm:interior L2alpha}(1) presented in
Subsection \ref{subsection:proof interior(1)}. Here we explain the changes that need to be made.
We follow parallel steps as those in the proof of Lemma \ref{bola fija}.
Indeed, by using Remarks \ref{remark1} and \ref{remark2} 
we can prove that there exist $0<\delta,\lambda<1$ such that
$$\frac{1}{\lambda^n}\int_{B_\lambda^+}|V(x,0)|^2\,dx+\frac{1}{\lambda^{n+1+a}}
\int_{(B_\lambda^+)^*}y^a|V|^2\,dX<\lambda^{2(\alpha+2s)}.$$
Notice that in the present case we have in Lemma \ref{bola fija} that $c=0$.
Now the rescaling process of Lemma \ref{lem:claim constant} comes into play,
but now we must take $c_k=0$ for all $k\geq0$. Every step
goes through by just changing balls by half balls, because when we rescale we always
get a normalized solution to \eqref{V problem}.

%%%%%%%%%%%%%%%%%%%%%%%%%%%%%%%%%%%%%%%%%%%%%%%%%%%%%%
\subsection{Proof of Theorem \ref{thm:boundary L2alpha}(2)}
%%%%%%%%%%%%%%%%%%%%%%%%%%%%%%%%%%%%%%%%%%%%%%%%%%%%%%

As in the proof of part (1), we just have to check that there exists $0<\delta<1$ such that
in the proper normalized situation for \eqref{V problem}
there is a linear function $\ell_\infty(x)=B_\infty\cdot x$ such that
$$\frac{1}{r^n}\int_{B_r^+}|V(x,0)-\ell_\infty(x)|^2\,dx\leq Cr^{2(\alpha+2s)},$$
for $r>0$ sufficiently small. Now we suppose that $F(0)=0$ and that we have the Campanato condition
$$\sup_{0<r\leq1}\frac{1}{r^{n+1+a+2(\alpha+2s-1)}}\int_{(B_r^+)^*}y^a|F|^2\,dX<\delta^2.$$
But again we can follow parallel
steps as those of the proof of Theorem \ref{thm:interior L2alpha}(2).
Observe that in our case the independent term $A$ in the linear function
that appears in Lemma \ref{a ver} is $0$ since for the approximating harmonic function
$\mathcal{W}$ we have $\mathcal{W}(0,0)=0$. This is essential in the iteration process
in order to always have a rescaled solution that is $0$ on $B_1\cap\{x_n=0\}$.
Further details are omitted.

%%%%%%%%%%%%%%%%%%%%%%%%%%%%%%%%%%%%%%%%%%%%%%%%%%%%%%
\subsection{Proof of Theorem \ref{thm:boundary Lmenos2alpha}}
%%%%%%%%%%%%%%%%%%%%%%%%%%%%%%%%%%%%%%%%%%%%%%%%%%%%%%

As before, it is enough to prove the result for $u(x)=U(x,0)$, where
$U\in H^1((B_1^+)^*,y^adX)$ is a solution to \eqref{dale que va},
where $F_{n+1}=0$. For part (1) we assume the Morrey condition
$$\sup_{0<r\leq1}\frac{1}{r^{n+1+a+2(\alpha-1)}}\int_{(B_r^+)^*}y^a|F|^2\,dX<\infty,$$
while for part (2) we suppose that $F(0)=0$ and that we have the Campanato condition
$$\sup_{0<r\leq1}\frac{1}{r^{n+1+a+2\alpha}}\int_{(B_r^+)^*}y^a|F|^2\,dX<\infty.$$
Now the proof follows exactly the same steps of the proof of Theorem \ref{thm:interior Lmenos2alpha},
by just replacing $B_r$ by $B_r^+$. We also observe that in this case the constant $c$ and the independent
term $A$ that come from the approximating harmonic function $\mathcal{W}$ are both $0$.
We omit further details.

%%%%%%%%%%%%%%%%%%%%%%%%%%%%%%%%%%%%%%%%%%%%%%%%%%%%%%
\section{The case of Neumann boundary condition}\label{section:7}
%%%%%%%%%%%%%%%%%%%%%%%%%%%%%%%%%%%%%%%%%%%%%%%%%%%%%%

In this section we sketch the proof of 
the results in the case of the Neumann problem \eqref{uman}.
Recall that the domain of $L_N$ is the Sobolev space $H^1(\Omega)$.
There exists an orthonormal basis of $L^2(\Omega)$ consisting
 of eigenfunctions $\varphi_k\in H^1(\Omega)$, $k=0,1,2,\ldots,$
that correspond to eigenvalues $0=\mu_0<\mu_1\leq\mu_2\leq\cdots\nearrow\infty$ of $L_N$. 
The domain of the fractional operator
$L_N^s$ is the Hilbert space $\mathcal{H}^s_N$ of functions
$u\in L^2(\Omega)$ such that $\int_\Omega u\,dx=0$ and $\sum_{k=1}^\infty\mu_k^su_k^2\equiv
\sum_{k=1}^\infty\mu_k^s|\langle u,\varphi_k\rangle_{L^2(\Omega)}|^2<\infty$. We define $L_N^su$ by
\eqref{definition Neumann} in the dual space $\mathcal{H}_N^{-s}:=(\mathcal{H}^s_N)'$. Notice that $\langle L_N^su,1\rangle=0$.
The heat semigroup generated by $L_N$ is given by
$$e^{-tL_N}u(x)=\sum_{k=0}^\infty e^{-t\mu_k}u_k\varphi_k(x)
=\int_\Omega W^N_t(x,z)u(z)\,dz,$$
where $W^N_t(x,z)$ is the corresponding heat kernel. Observe that $e^{-tL_N}1(x)=1$ for all $x\in\Omega$, $t>0$.
As in Theorem \ref{thm:puntual} we can prove that \eqref{puntual Neumann} holds
for $u,\psi\in\mathcal{H}^s_N$ and 
$$K_s^N(x,z):=\frac{1}{2|\Gamma(-s)|}\int_0^\infty W_t^N(x,z)\,\frac{dt}{t^{1+s}}.$$
It is well known that if $\Omega$ is an exterior domain or the region lying above the graph of a
Lipschitz function, then the heat kernel $W^N_t(x,z)$ as global upper Gaussian estimates. If
the domain is bounded then the Gaussian estimate holds only for short times
 and the heat kernel is bounded in $x$, $z$ and $t$ as $t\to\infty$. For this see \cite{Auscher}
 and the references therein.
Hence, as we did in Section \ref{section:2} we can prove that the kernel $K^N_s(x,z)$ satisfies the estimate
$$0\leq K_s^N(x,z)\leq\frac{c_{\Omega,n,s}}{|x-z|^{n+2s}},\quad x,z\in\Omega.$$
For the heat kernel related to the Neumann Laplacian we have two-sided Gaussian estimates
(see \cite{Saloff-Coste}, also \cite{Stinga-Volzone}, and the references therein),
which imply that in this case the estimate for $K_s^N(x,z)$ from above also holds from below with a different constant.
The extension problem for $L_N^s$ is given as follows. The solution $U$ to
\begin{equation}\label{extension Neumann}
\begin{cases}
\dive_x(A(x)\nabla_xU)+\tfrac{a}{y}U_y+U_{yy}=0,&\hbox{in}~\Omega\times(0,\infty),\\
\partial_AU(x,y)=0,&\hbox{on}~\partial\Omega\times[0,\infty),\\
U(x,0)=u(x),&\hbox{on}~\Omega,
\end{cases}
\end{equation}
satisfies, for $c_s=|\Gamma(-s)|/(4^s\Gamma(s))>0$,
$$-\frac{1}{2s}\lim_{y\to0^+}y^{a}U_y(x,y)=-\lim_{y\to0^+}\frac{U(x,y)-U(x,0)}{y^{2s}}=c_sL^s_Nu(x),\quad
\hbox{in}~\mathcal{H}_N^{-s},$$
see \cite{Stinga,Stinga-Torrea}.
Similar scaling properties as those of Section \ref{section:2} hold for $L_N^s$.
Parallel to Theorem \ref{thm:LSW},
the fundamental solution $G_s^N(x,z)$ of $L^s_N$ verifies the interior estimate
$$G_s^N(x,z)\sim\frac{c_{\Omega,n,s}}{|x-z|^{n-2s}},\quad n\neq 2s,$$
and it is logarithmic when $n=2s$. The interior Harnack inequality for $L^s_N$ is also true.

Next we show that the domain of $L_N^s$ is the fractional Sobolev space $H^s(\Omega)$,
which is the closure of $C^\infty(\Omega)$ with respect to the norm
$\|u\|_{H^s(\Omega)}^2=\|u\|^2_{L^2(\Omega)}+[u]^2_{H^s(\Omega)}$.
Notice that the solution $U$ to the extension problem \eqref{extension Neumann}
belongs to $H^1(\Omega\times(0,\infty),y^adX)$ and that for each fixed $y\geq0$ we have
$\int_\Omega U(x,y)\,dx=0$. Since the trace of $U$ is an element of 
the fractional Sobolev space $H^s(\Omega)$ (see for example \cite{Lions-Magenes}),
it turns out that $\mathcal{H}^s_N\subset H^s(\Omega)$.
For the other inclusion, we notice that the energy for the extension problem for $L_N^s$ is comparable
to the energy for the extension problem for the Neumann Laplacian $-\Delta_N$. This and the following Lemma give that
if $\int_\Omega u\,dx=0$ then $u\in\mathcal{H}_N^s$ if and only if $u\in H^s(\Omega)$.

\begin{lem}
Let $u:\Omega\to\R$ such that $\int_\Omega u\,dx=0$.
For any $0<s<1$, $u\in \Dom((-\Delta_N)^s)$ if and only if $u\in H^s(\Omega)$. 
In this case we have
$$\|(-\Delta_N)^{s/2}u\|_{L^2(\Omega)}\sim [u]_{H^s(\Omega)}.$$
\end{lem}

\begin{proof}
The idea is similar to the proof of Theorem 2.4 in \cite{Stinga-Volzone} and here we sketch the steps.
Let $U$ be the solution to the extension problem for the fractional Neumann Laplacian in $\Omega$. Then
$$-\frac{\langle U(\cdot,y),u\rangle_{L^2(\Omega)}-\langle u,u\rangle_{L^2(\Omega)}}{y^{2s}}\to
c_s\langle (-\Delta_N)^su,u\rangle_{L^2(\Omega)}=c_s\|(-\Delta_N)^{s/2}u\|^2_{L^2(\Omega)},$$
as $y\to0^+$. Now, with the Poisson kernel $P_y^{s,N}$ for the Neumann Laplacian extension problem
(as in \eqref{Poisson formula}--\eqref{Poisson kernel}
but with the Neumann heat kernel in place of $W_t(x,z)$) we get
$$\langle U(\cdot,y),u\rangle_{L^2(\Omega)}-\langle u,u\rangle_{L^2(\Omega)}
=\int_\Omega\int_\Omega P_y^{s,N}(x,z)\big(u(z)u(x)-u(x)^2\big)\,dx\,dz.$$
Here we used the fact that
$$\int_\Omega P_y^{s,N}(x,z)\,dz=1,\quad\hbox{for all}~x\in\Omega,t>0.$$
By exchanging the roles of $x$ and $z$ and using that $P_y^{s,N}(x,z)=P_y^{s,N}(z,x)$, we get
$$\langle U(\cdot,y),u\rangle_{L^2(\Omega)}-\langle u,u\rangle_{L^2(\Omega)}
=\frac{1}{2}\int_\Omega\int_\Omega\big(u(z)-u(x)\big)
P_y^{s,N}(x,z)\,dz\,dx.$$
But now, since the heat kernel for the Neumann Laplacian has two sided Gaussian estimates
\cite{Saloff-Coste, Stinga-Volzone},
we readily get
$$P_y^{s,N}(x,z)\sim\frac{y^{2s}}{(|x-z|^2+y^2)^{(n+2s)/2}}.$$
Therefore, by dividing by $y^{2s}$
and taking $y\to0^+$ above we arrive to $\|(-\Delta_N)^{s/2}u\|_{L^2(\Omega)}\sim[u]_{H^s(\Omega)}$.
\end{proof}

Given $f\in\mathcal{H}^{-s}_N(\Omega)$ (observe that $\langle f,1\rangle=0$), there exists
a unique solution $u\in H^s(\Omega)$ to \eqref{uman} with $\int_\Omega u\,dx=0$.

It is clear that the interior regularity results of Theorems \ref{thm:interior Calpha} and \ref{thm:interior Lp}
hold for the case of the fractional Neumann operator $L^s_N$. Indeed, the proof 
presented in Section \ref{section:4} is based on a Caccioppoli
estimate for the (localized)
 extension problem and the Dirichlet boundary condition plays no significant role there.

The boundary results deserve some attention.
The reason why the interior regularity should hold up to the boundary 
becomes apparent once we look for the half space case.
 Consider the 
fractional Neumann Laplacian $-\Delta_N^+$ in a half space $\R^n_+$. For $u$ having mean zero we have
$$(-\Delta_N^+)^su(x)=\frac{1}{\Gamma(-s)}\int_0^\infty\big(e^{t\Delta_N^+}u(x)-u(x)\big)
\,\frac{dt}{t^{1+s}},\quad x\in\R^n_+.$$
We denote by $u_e$ the even reflection of $u$ with respect to the variable $x_n$:
$$u_e(x)=\begin{cases}
u(x),&\hbox{if}~x_n\geq0,\\
u(x^*),&\hbox{if}~x_n<0.
\end{cases}$$
Then the method of reflections give $e^{-t\Delta_N^+}u(x)=e^{t\Delta}u_e(x)$, for $x\in\R^n_+$. Therefore,
$$(-\Delta_N^+)^{s}u(x)=c_{n,s}\int_{\R^n_+}\big(u(x)-u(z)\big)
  \left(\frac{1}{|x-z|^{n+2s}}+\frac{1}{|x-z^\ast|^{n+2s}}\right)\,dz,\quad x\in\R^n_+.$$
Now, we easily see that if $(-\Delta_N^+)^su(x)=f(x)$, with $f$ having zero mean in $\R^n_+$, then $(-\Delta)^su_e(x)=f_e(x)$,
for $x\in\R^n$. As a conclusion, by applying Propositions \ref{Prop:Silvestre1} and \ref{Prop:Silvestre2} we get the following.

\begin{thm}[Global regularity in half space -- Neumann case]\label{Thm:Laplacian half space Neumann}
Let $u$ be a bounded solution to
$$\begin{cases}
   (-\Delta_N^+)^su=f,&\hbox{in}~\R^n_+,\\
   \partial_{x_n}u=0,&\hbox{on}~\partial\R^n_+=\{x_n=0\},
  \end{cases}$$
  where $f$ has zero mean.
\begin{itemize}
\item Suppose that $f\in C^{0,\alpha}(\overline{\R^n_+})$, $0<\alpha\leq1$. Then
\begin{itemize}
 \item[(1)] If $\alpha+2s<1$ then $u\in C^{0,\alpha+2s}(\overline{\R^n_+})$ and 
 $$\|u\|_{C^{0,\alpha+2s}(\overline{\R^n_+})}\leq C\big(\|u\|_{L^\infty(\R^n_+)}+\|f\|_{C^{0,\alpha}(\overline{\R^n_+})}\big).$$
 \item[(2)] If $1<\alpha+2s<2$ then $u\in C^{1,\alpha+2s-1}(\overline{\R^n_+})$ and
 $$\|u\|_{C^{1,\alpha+2s-1}(\overline{\R^n_+})}\leq C\big(\|u\|_{L^\infty(\R^n_+)}+\|f\|_{C^{0,\alpha}(\overline{\R^n_+})}\big).$$
 \item[(3)] If $2<\alpha+2s<3$ then $u\in C^{2,\alpha+2s-2}(\overline{\R^n_+})$ and
 $$\|u\|_{C^{2,\alpha+2s-2}(\overline{\R^n_+})}\leq C\big(\|u\|_{L^\infty(\R^n_+)}+\|f\|_{C^{0,\alpha}(\overline{\R^n_+})}\big).$$
\end{itemize}
\item Suppose that $f\in L^\infty(\R^{n}_+)$. Then
\begin{itemize}
\item[(i)] If $0<2s<1$ then $u\in C^{0,2s}(\overline{\R^n_+})$ and
 $$\|u\|_{C^{0,2s}(\overline{\R^n_+})}\leq C\big(\|u\|_{L^\infty(\R^n_+)}+\|f\|_{L^\infty(\R^n_+)}\big).$$
 \item[(ii)] If $2s=1$ then $u$ is in the Zygmund space $\Lambda_\ast(\overline{\R^n_+})$ and
 $$\|u\|_{\Lambda_\ast(\overline{\R^n_+})}\leq C\big(\|u\|_{L^\infty(\R^n_+)}+\|f\|_{L^\infty(\R^n_+)}\big).$$
 \item[(iii)] If $1<2s<2$ then $u\in C^{1,2s-1}(\overline{\R^n_+})$ and
 $$\|u\|_{C^{1,2s-1}(\overline{\R^n_+})}\leq C\big(\|u\|_{L^\infty(\R^n_+)}+\|f\|_{L^\infty(\R^n_+)}\big).$$
\end{itemize}
\end{itemize}
All the constants $C$ above depend only on $n$, $\alpha$ and $s$.
\end{thm}

For the general case, as we did before for the Dirichlet case,
it is enough to prove the regularity at the origin for the solution $U$
to the extension problem
$$\begin{cases}
\dive(y^aB(x)\nabla U)=\dive(y^aF),&\hbox{in}~(B_1^+)^*,\\
-y^aU_y\big|_{y=0}=f,&\hbox{on}~B_1^+,\\
\partial_AU(x,y)=0,&\hbox{on}~B_1^*\cap\{x_n=0\},
\end{cases}$$
where $A$ and $F$ have the corresponding regularity.
We can assume that $B(0)=I$, so that $\partial_AU(0,0)=-\partial_{x_n}U(0,0)=0$.
The same Caccioppoli estimate holds in this case.
The approximation lemma can then be proved to
get an approximating harmonic
function $\mathcal{W}\in H^1((B_{3/4}^*)^+,y^adX)$,
see also Remark \ref{remark1}.
By performing an even reflexion instead of an odd one we can reproduce 
the argument in Remark \ref{remark2} and conclude that $\mathcal{W}$ has the desired regularity.
From here on we can repeat the arguments given in Section \ref{section:6}.
Details are left to the interested reader.

%%%%%%%%%%%%%%%%%%%%%%%%%%%%%%%%%%%%%%%%%%%%%%%%%%%%%%

%%%%%%%%%%%%%%%%%%%%%%%%%%%%%%%%%%%%%%%%%%%%%%%%%%%%%%

%%%%%%%%%%%%%%%%%%%%%%%%%%%%%%%%%%%%%%%%%%%%%%%%%%%%%%
\end{document}